\documentclass[smallextended, envcountsect]{svjour3}
\usepackage{makeidx}
\usepackage{graphicx}
\usepackage{amsmath}
\usepackage{amssymb}
\usepackage{times, color}
\usepackage{graphicx,epsfig}
\usepackage{exscale, fancybox}
\usepackage[applemac]{inputenc}
\usepackage{latexsym}

\RequirePackage{fix-cm}
\smartqed
\journalname{Preprint version}

  \definecolor{DarkOrchid} {cmyk}{0.40,0.80,0.20,0}
\definecolor{VioletRed} {cmyk}{0,0.81,0,0}
\definecolor{Orange} {cmyk}{0,0.61,0.87,0}

\newcommand{\dom}{\operatorname{dom}}
\newcommand{\epi}{\ensuremath{\operatorname{epi}}}

\newtheorem{Thm}{Theorem}[section]
\newtheorem{Cor}{Corollary}[section]
\newtheorem{Def}{Definition}[section]
\newtheorem{Exa}{Example}[section]
\newtheorem{Lem}{Lemma}[section]
\newtheorem{Pro}{Proposition}[section]
\newtheorem{Rem}{Remark}[section]
\begin{document}

\title{An additive subfamily of enlargements of a maximally monotone operator%
}
\dedication{Dedicated to professor L.~Thibault}

\author{Regina S. Burachik  \and Juan Enrique Mart\'\i nez-Legaz \thanks{ The
research of Juan-Enrique Mart{í}nez-Legaz was supported by the MINECO of
Spain, Grant MTM2011-29064-C03-01, and the Australian Research Council,
project DP140103213. He is affiliated to MOVE (Markets, Organizations and
Votes in Economics).}\and %
Mahboubeh Rezaie \and Michel Th\' era\thanks{The research of Michel Théra was partially supported by the MINECO of Spain,
Grant MTM2011-29064-C03-03, and the Australian Research Council, project
DP110102011.}}
\institute{First author \at School of Information Technology and  Mathematical Sciences,
 University of South Australia, Mawson Lakes, SA 5095, Australia\\
              \email{regina.burachik@unisa.edu.au}                      \and
          Second author \at
             Departament d'Economia i
  d'Historia Economica, Universitat Autonoma de Barcelona,
  Spain\\
 \email{Juan-Enrique.Martinez-Legaz@uab.cat}
 \and
Third author\at
            University of Isfahan, Iran\\
 \email{mrezaie@sci.ui.ac.ir}
 \and
Fourth author \at
            Universit\' e de Limoges, France and  Centre for Informatics and Applied Optimisation,  Federation University  Australia\thanks{MTM2011-29064-C03(03) from MINECO}\\
 \email{michel.thera@unilim.fr}
 }

\date{Received: date / Accepted: date}
\maketitle

\begin{abstract}
We introduce a subfamily of additive enlargements of a maximally monotone
operator. Our definition is inspired by the early work of Simon Fitzpatrick. These enlargements constitute  a subfamily of the family of
enlargements  introduced by Svaiter.  When the  operator  under consideration is the subdifferential of a convex lower semicontinuous proper function, we prove that some members of the subfamily are
 \emph{%
smaller} than the classical ${\varepsilon }$-subdifferential enlargement widely used in convex analysis. We also
recover the epsilon-subdifferential within the subfamily. Since they are all
additive, the enlargements in our subfamily can be seen as structurally
closer to the $\varepsilon $-subdifferential enlargement. 
\end{abstract}


\subtitle{Enlargements of a maximally monotone operator}





\section{Introduction}

\label{intro} Let $X$ be a real Banach space with continuous dual $X^{\ast }$%
. By a generalized equation governed by a maximally monotone operator $%
T:X\rightrightarrows X^{\ast }$, we mean the problem of finding $x\in X\;%
\text{such that}$ 
\begin{equation}
0\in T(x).  \label{A}
\end{equation}%
This model has been extensively used as a mathematical formulation of
fundamental problems in optimization and fixed point theory. Main
illustrations follow.

\begin{itemize}
\item If $X$ is a Hilbert space, $I$ is the identity map, $%
F:X\longrightarrow X$ is a nonexpansive mapping, and $T = I - F$, then
solving (\ref{A}) is equivalent to finding a fixed point of $F$.

\item If $T$ is a maximally monotone operator from a Hilbert space into
itself, then the set of solutions of (\ref{A}) is the set of fixed points of
the so-called resolvent map $R:=(I+\lambda T)^{-1}$, with $\lambda >0,$ or
the set of fixed points of the Cayley operator $C:=2R-I$.

\item As observed by Rockafellar \cite[Theorem 37.4]{roc70}, when $L:X\times
X\to \mathbb{R}$ is a concave-convex function (for instance the Lagrangian
of a convex program), finding a saddle point of $L$ is equivalent to solving 
$(0,0)\in \partial L(x,y)$, where $\partial
L(x,y)=\partial_{x}(-L)(x,y)\times\partial_{y}L(x,y)$, and $\partial_x$ and $%
\partial_y$ are the convex subdifferentials operators with respect to the
first and the second variable, respectively.

\item If $f$ is a  lower semicontinuous proper convex function and $%
T=\partial f$, the subdifferential of $f$, then the set of solutions of (\ref%
{A}) is the set of minimizers of $f$.
\end{itemize}

Solving inclusion (\ref{A}) is tantamount to finding a point of the form $%
(x,0)$ in the graph of $T$. If $T$ is not point-to-point, then it lacks
semicontinuity properties. Namely, if $Tx$ is not a singleton, then $T$
cannot be inner-semicontinuous at $x$ (see \cite[Theorem 4.6.3]{BI}). This
fact makes the problem ill-behaved, making the required computations hard.
Enlargements of $T$ are point-to-set mappings (the terms set-valued mapping
and multifunction are also used) which have a graph larger than the graph of 
$T$. These mappings, however, have better continuity properties than $T$
itself. Moreover, they stay \textquotedblleft close" to $T$, so they allow
to define perturbations of problem \eqref{A}, without losing information on $%
T$. In this way, we can define well-behaved approximations of problem %
\eqref{A}, which (i) are numerically more robust, and (ii) whose solutions
approximate accurately the solutions of \eqref{A}. The use of enlargements
in the study of problem \eqref{A} has been a fruitful approach, from both
practical and theoretical reasons. A typical example of the usefulness of
enlargements in the analysis of \eqref{A} arises when considering a convex
optimization problem, i.e., the case in which $T=\partial f$, where $%
f:X\rightarrow \mathbb{R}\cup \{+\infty \}$ is a proper, convex, and lower
semicontinuous 
 function. It is a well-known fact that $%
T=\partial f$ is maximally monotone. This has been proved by Moreau for
Hilbert spaces \cite{moreau} and by Rockafellar \cite{R2} for Banach spaces.
The \emph{$\epsilon $-subdifferential of }$f$, introduced by Brøndsted and
Rockafellar in \cite{BR} (see Definition \ref{eps-subdiff:def}), is an
enlargement of $T=\partial f$ which had a crucial role in the development of
algorithms for solving \eqref{A}, as well as in allowing a better
understanding of the properties of the mapping $\partial f$ itself (see,
e.g., \cite{R2}). This is why the $\epsilon $-subdifferential has been
intensively studied since its introduction in 1965, not only from an
abstract point of view, but also for constructing specific numerical methods
for convex nonsmooth optimization (see, e.g., \cite{AIS,Kiw1,Kiw2,LeS,ScZ}).
Using the optimization problem as a benchmark, but having the general
problem \eqref{A} in mind, it is relevant to study enlargements of an
arbitrary maximally monotone operator $T$. To be useful, the enlargements of 
$T$ must share with the $\epsilon $-subdifferential most of its good
properties. By good properties we mean local boundedness, demi-closedness of
the graph, Lipschitz continuity, and Brøndsted-Rockafellar property. Indeed,
given an arbitrary maximally monotone operator $T$ defined on a reflexive
Banach space, Svaiter introduced in \cite{SV} a family of enlargements,
denoted by $\mathbb{E}(T)$, which share with the $\epsilon $-subdifferential
all these good properties. There are, however, properties of the $\epsilon $%
-subdifferential which are not shared by every element of $\mathbb{E}(T)$.
To make this statement precise, we recall the largest member of the family $%
\mathbb{E}(T)$, denoted by $T^{\mathrm{BE}}$. The enlargement $T^{\mathrm{BE}%
}:\mathbb{R}_{+}\times X\rightrightarrows X^{\ast }$ has been the intense
focus of research (see, e.g. \cite{BI,BIS,BSS1,BS99,BS02,ML96,RT,SV}), and
is defined as follows. We say that%
\begin{equation}
x^{\ast }\in T^{\mathrm{BE}}(\epsilon ,x)\,\iff\, \forall (y,y^{\ast })\in%
\mathop{\rm gph\,} T\,\text{ we have }\langle y-x,y^{\ast }-x^{\ast }\rangle
\geq -\epsilon .  \label{teps}
\end{equation}%
The discrepancy between some elements of $\mathbb{E}(T)$ and the $\epsilon $%
-subdifferential arises from the fact that, when $T=\partial f$, the biggest
enlargement $T^{\mathrm{BE}}$ is larger than the $\epsilon $%
-subdifferential. Namely, ${\partial }_{\epsilon }f(\cdot )\subset ({%
\partial }f)^{BE}(\epsilon ,\cdot )$, and the inclusion can be strict, as
noticed by Martínez-Legaz and Théra, see \cite{ML96}. Hence, it is natural
to expect that some properties of the $\epsilon $-subdifferential \emph{will
not} be shared by every element of $\mathbb{E}(T)$, and in particular, they
will not be shared by $T^{\mathrm{BE}}$. Such a property is \emph{additivity}%
. In the context of enlargements of arbitrary maximally monotone operators,
this property was introduced in \cite{BS99} and further studied in \cite%
{SV,S03}. It is stated as follows. An enlargement $E:\mathbb{R}_{+}\times
X\rightrightarrows X^{\ast }$ is \emph{additive} if for every $x_{1}^{\ast
}\in E(${$\epsilon $}${_{1}},x_{1})$ and every $x_{2}^{\ast }\in E(${$%
\epsilon $}${_{2}},x_{2})$, it holds that 
\begin{equation*}
\langle x_{1}-x_{2},x_{1}^{\ast }-x_{2}^{\ast }\rangle \geq -(\epsilon
_{1}+\epsilon _{2}).
\end{equation*}%
The $\epsilon $-subdifferential is additive. Moreover, it is \emph{maximal}
among all those enlargements of $\partial f$ with this property. In other
words, if another enlargement of $\partial f$ is additive and contains the
graph of the $\epsilon $-subdifferential, then it must coincide with the $%
\epsilon $-subdifferential enlargement. We describe the latter property as
being \emph{maximally additive} (or \emph{max-add}, for short). Namely, an
enlargement $E$ is max-add when it is additive and, if the graph of another
additive enlargement $E^{\prime }$ contains the graph of $E$, then we must
have $E=E^{\prime }$. Since the $\epsilon $-subdifferential is max-add, the
members of the family $\mathbb{E}(T)$ that are max-add do share an extra
property with the $\epsilon $-subdifferential, and in this sense, they can
be seen as structurally \textquotedblleft closer\textquotedblright\ to the $%
\epsilon $-subdifferential. As hinted above, not all enlargements $E\in 
\mathbb{E}(T)$ are additive. However, it was proved in \cite{SV} that the
smallest enlargement, denoted by $T^{\mathrm{SE}}$, is additive. The
existence of a max-add element in $\mathbb{E}(T)$ is then obtained in \cite%
{SV} as a consequence of Zorn's lemma. In the present paper, we define a
whole family of additive elements of $\mathbb{E}(T)$, denoted by $\mathbb{E}%
_{\mathcal{H}}(T)$. The family $\mathbb{E}_{\mathcal{H}}(T)$ has max-add
elements, and the existence of these elements is deduced through a
constructive proof. For the case in which $T=\partial f$, we show that some
specific elements of $\mathbb{E}_{\mathcal{H}}(T)$ are contained in the $%
\epsilon $-subdifferential enlargement. Additionally, a specific element of
our family coincides with the $\epsilon $-subdifferential when $T=\partial f$%
.

The layout of the paper is as follows. First, we define our family of
enlargements of a maximally monotone operator $T$. Our definition is
inspired by the early work of Fitzpatrick presented in \cite{F}, but can as
well be seen as a subfamily of $E(T)$. Second, we prove that all members of
our subfamily are additive. We also introduce a new definition related to
additivity, which helps us in the proofs. We deduce, in a constructive way,
the existence of max-add elements in $\mathbb{E}(T)$. Finally, we consider
the case $T=\partial f$. For this case we prove that some members of the
subfamily are \emph{smaller} than the {$\epsilon $}-subdifferential
enlargement, and we recover the $\epsilon $-subdifferential as a member of
our subfamily.


\section{Basic Definitions}

Throughout this paper, we assume that $X$ is a real \textbf{reflexive}
Banach space with continuous dual $X^{\ast }$, and pairing between them
denoted by $\langle \cdot ,\cdot \rangle .$ We will use the same symbol $%
\Vert \cdot \Vert $ for the norms in $X$ and $X^{\ast }$, and $w$ will stand
for the weak topologies on $X$ and $X^{\ast }$. We consider the Cartesian
product $X\times X^{\ast }$ equipped with the product topology determined by
the norm topology in $X$ and the weak topology in $X^{\ast }.$ In this case
the dual of $X\times X^{\ast }$ can be identified with $X^{\ast }\times X$
and hence, the dual product is defined as $\langle (x,x^{\ast }),(y^{\ast
},y)\rangle =\langle x,y^{\ast }\rangle +\langle y,x^{\ast }\rangle .$

For a given (in general, multivalued) operator $T:X\rightrightarrows X^{\ast
},$ its graph is denoted by 
\begin{equation*}
\mathop{\rm gph\,}(T):=\{(x,x^{\ast })\in X\times X^{\ast }:x^{\ast }\in
T(x)\}.
\end{equation*}%
%
%
%
%
%
%
%
%
%
%
%
%
%
%
Recall that $T:X\rightrightarrows X^{\ast }$ is said to be \emph{monotone}
if and only if 
\begin{equation*}
\langle y-x,y^{\ast }-x^{\ast }\rangle \geq 0\hspace{1cm}\forall (x,x^{\ast
}),(y,y^{\ast })\in \mathop{\rm gph\,}(T).
\end{equation*}%
A monotone operator $T$ is called \emph{maximally monotone } if and only if
the condition $\langle y-x,y^{\ast }-x^{\ast }\rangle \geq 0$ for every $%
(y,y^{\ast })\in \mathop{\rm gph\,}(T),$ implies $(x,x^{\ast })\in %
\mathop{\rm gph\,}(T)$. Equivalently, it amounts to saying that $T$ has no
monotone extension (in the sense of graph inclusion).

In what follows, $f:X\rightarrow \mathbb{R}\cup \{+\infty \}$ will be a
convex function. Recall that $f$ is \textit{proper} if the set $\mathrm{dom}%
\left( f\right) :=\{x\in X:f(x)<+\infty \}$ is nonempty. The \emph{%
subdifferential} of $f$ is the multivalued mapping $\partial
f:X\rightrightarrows X^{\ast }$ defined by 
\begin{equation}
{\partial }f(x):=\{x^{\ast }\in X^{\ast }:f(y)-f(x)\geq \langle y-x,x^{\ast
}\rangle ,\,\forall y\in X\},  \label{subdiff:def}
\end{equation}%
if $x\in \mathrm{dom}\left( f\right) $, and ${\partial }f(x):=\emptyset $,
otherwise. Given $\epsilon \geq 0,$ the $\epsilon $-\emph{subdifferential of}
$f$ is the multivalued mapping $\partial _{\epsilon }f:X\rightrightarrows
X^{\ast }$ defined by 
\begin{equation}
{\partial }_{\epsilon }f(x):=\{x^{\ast }\in X^{\ast }:f(y)-f(x)\geq \langle
y-x,x^{\ast }\rangle -\epsilon ,\,\forall y\in X\},  \label{eps-subdiff:def}
\end{equation}%
if $x\in \mathrm{dom}\left( f\right) $, and ${\partial }_{\epsilon
}f(x):=\emptyset $, otherwise. The case $\epsilon =0$ gives the
subdifferential of $f$ at $x$. The set $\partial _{\epsilon }f(x)$ is
nonempty for every $\epsilon >0$ if and only if $f$ is lower semicontinuous  at $x$. Note
that the $\epsilon $-subdifferential can be viewed as an approximation of
the subdifferential. Indeed, in \cite[Theorem 1]{ML96} a formula expressing,
for a   lower semicontinuous  convex extended-real-valued function, its $\epsilon -$%
subdifferential in terms of its subdifferential was established.

As we will see later in Subsection 2.1, enlargements are multifunctions
defined on $\mathbb{R}_{+}\times X$. Consequently, we need a different
notation for the epsilon-subdifferential \eqref{eps-subdiff:def}. This
enlargement will be denoted as follows: 
\begin{equation*}
\breve{{\partial }}f(\epsilon ,x):=\partial _{\epsilon }f(x).
\end{equation*}%
We call the enlargement $\breve{{\partial }}f$ the \emph{\ Brø%
ndsted-Rockafellar enlargement of} ${\partial }f.$ The Fenchel-Moreau
conjugate of $f$ is denoted by $f^{\ast }:X^{\ast }\rightarrow \mathbb{R}%
\cup \{+\infty \}$ and is defined by 
\begin{equation}
f^{\ast }(x^{\ast }):=\sup \{\langle x,x^{\ast }\rangle -f(x):x\in X\}.
\label{eq:1}
\end{equation}%
Observe that $f^{\ast }$ is  lower semicontinuous  with respect to the weak topology on $%
X^{\ast }.$ In what follows, we shall denote by $f^{FY}$ the Fenchel-Young
function associated to $f$: 
\begin{equation*}
f^{FY}(x,x^{\ast }):=f(x)+f^{\ast }(x^{\ast })\;\text{for all }\;(x,x^{\ast
})\in X\times X^{\ast }.
\end{equation*}%
Then $f^{FY}$ is a convex, proper and $(\Vert \cdot\Vert \times w)$-lower semicontinuous 
function on $X\times X^{\ast }$ and it is well known that $f^{FY}$
completely characterizes the graph of the subdifferential of $f$: 
\begin{equation}
\partial f(x)=\{x^{\ast }\in X^{\ast }:f^{FY}(x,x^{\ast })=\langle x,x^{\ast
}\rangle \}.  \label{eq:9}
\end{equation}%
Moreover, $f^{FY}$ also completely characterizes the graph of the Brø%
ndsted-Rockafellar enlargement of $\partial f$. Namely, 
\begin{equation}
x^{\ast }\in \breve{{\partial }}f(\epsilon ,x)\hbox{ if and only if }%
f^{FY}(x,x^{\ast })\leq \langle x,x^{\ast }\rangle +\epsilon .  \label{eq:10}
\end{equation}%
If $Z$ is a general Banach space and $f,g:Z\rightarrow \mathbb{R}\cup
\{+\infty \}$, the \emph{infimal convolution} of $f$ with $g$ is denoted by $%
f\oplus g$ and defined by 
\begin{equation*}
(f\oplus g)(z):=\inf_{{\tiny z_{1}+z_{2}=z}}\{f(z_{1})+g(z_{2})\}.
\end{equation*}

If $q:Z\rightarrow \mathbb{R}\cup \{+\infty \}$, the \emph{\ closure} of $q$
is denoted by $\mathrm{cl}\,(q)$ and defined by: 
\begin{equation*}
\epi (\mathrm{cl}\,(q))=\mathrm{cl}\,(\epi (q)).
\end{equation*}%
We will use the following well-known property: 
\begin{equation}
(f+g)^{\ast }=\mathrm{cl}\,\mathit{(f^{\ast }\oplus g^{\ast })}\leq \mathit{%
(f^{\ast }\oplus g^{\ast })}.  \label{eq:12}
\end{equation}

\subsection{The family $\mathbb{E}(T)$}

We mentioned above two examples of enlargements, the enlargement $\breve{%
\partial}f$ of $T=\partial f$, and the enlargement $T^{\mathrm{BE}}$ of an
arbitrary maximally monotone operator. Each of these is a member of a family
of enlargements of $\partial f$ and $T$, respectively. For a maximally
monotone operator $T$, denote by $\mathbb{E}(T)$ the following family of
enlargements defined in \cite{SV} and \cite{BS02}.

\begin{Def}
\label{def:enl-fam} Let $T:X\rightrightarrows X^{\ast }$. We say that a
point-to-set mapping\/ $E:\mathbb{R}_{+}\times X\rightrightarrows X^{\ast }$
belongs to the family%
\index{enlargement!family of} $\mathbb{E}(T)$ when

\begin{itemize}
\item[$(E_{1})$] $T(x)\subset E(\epsilon ,x)$ for all $\epsilon \geq 0,x\in
X $;

\item[$(E_{2})$] If $0\leq \epsilon _{1}\leq \epsilon _{2}\,$, then $%
E(\epsilon _{1},x)\subset E(\epsilon _{2},x)$ for all $x\in X$;

\item[$(E_{3})$] The transportation formula holds for $E$. More precisely,
let $x_{1}^{\ast }\in E(${$\epsilon $}$_{1},x_{1})$, $x_{2}^{\ast }\in E(${$%
\epsilon $}$_{2},x_{2}),$ and $\alpha \in \lbrack 0,1]$. Define 
\begin{equation*}
\hat{x}:=\alpha x_{1}+(1-\alpha )x_{2},
\end{equation*}%
\begin{equation*}
\hat{x}^{\ast }:=\alpha x_{1}^{\ast }+(1-\alpha )x_{2}^{\ast },
\end{equation*}%
\begin{eqnarray*}
\epsilon &:=&\alpha {\epsilon }_{1}+(1-\alpha )\epsilon _{2}+\alpha \langle
x_{1}-\hat{x},x_{1}^{\ast }-\hat{x}^{\ast }\rangle +(1-\alpha )\langle x_{2}-%
\hat{x},x_{2}^{\ast }-\hat{x}^{\ast }\rangle \\
&=& \alpha {\epsilon }_{1}+(1-\alpha )\epsilon _{2}+\alpha (1-\alpha)\langle
x_{1}-{x_2},x_{1}^{\ast }-x_{2}^{\ast }\rangle .
\end{eqnarray*}
\end{itemize}

Then $\epsilon \geq 0$ and $\hat{x}^{\ast }\in E(\epsilon ,\hat{x})$.
\end{Def}

The following lemma, which is well-known but hard to track down, states that
the transportation formula\ holds for the Brøndsted-Rockafellar enlargement.
We include its simple proof here for convenience of the reader.

\begin{Lem}
\label{TF} Let $f:X\rightarrow \mathbb{R}\cup \{+\infty \}$ be convex. Then
the transportation formula holds for $\breve{\partial}f.$
\end{Lem}

\begin{proof}
Assume that $x_{1}^{\ast }\in \breve{\partial}f(\epsilon _{1},x_{1})$, $%
x_{2}^{\ast }\in \breve{\partial}f(\epsilon _{2},x_{2})$ and $\alpha \in
\lbrack 0,1],$ and let $\hat{x},$ $\hat{x}^{\ast }$ and $\epsilon $ be as in
condition $(E_{3})$ of Definition \ref{def:enl-fam}. Let us first show that $%
\epsilon \geq 0$. By assumption, we have 
\begin{equation*}
\begin{array}{rcl}
f(x_{2})-f(x_{1}) & \geq & \langle x_{2}-x_{1},x_{1}^{\ast }\rangle
-\epsilon _{1}, \\ 
f(x_{1})-f(x_{2}) & \geq & \langle x_{1}-x_{2},x_{2}^{\ast }\rangle
-\epsilon _{2}.%
\end{array}%
\end{equation*}%
Summing up these inequalities and re-arranging the resulting expression
gives 
\begin{equation*}
\langle x_{1}-x_{2},x_{1}^{\ast }-x_{2}^{\ast }\rangle \geq -\epsilon
_{1}-\epsilon _{2}.
\end{equation*}%
We can now write 
\begin{equation*}
\alpha (1-\alpha )\langle x_{1}-x_{2},x_{1}^{\ast }-x_{2}^{\ast }\rangle
\geq \alpha (1-\alpha )(-\epsilon _{1}-\epsilon _{2})\geq -\alpha \epsilon
_{1}-(1-\alpha )\epsilon _{2}.
\end{equation*}%
Using the definition of $\epsilon$ in $(E_{3})$ , we deduce that $\epsilon
\geq 0$. In order to finish the proof, we use the assumption on $x_{1}^{\ast
},x_{2}^{\ast }$ to write 
\begin{equation*}
\begin{array}{rcl}
\alpha \left( f(z)-f(x_{1})\right) & \geq & \alpha \left( \langle
z-x_{1},x_{1}^{\ast }\rangle -\epsilon _{1}\right) \\ 
(1-\alpha )\left( f(z)-f(x_{2})\right) & \geq & (1-\alpha )\left( \langle
z-x_{2},x_{2}^{\ast }\rangle -\epsilon _{2}\right) . \\ 
&  & 
\end{array}%
\end{equation*}%
Summing up these inequalities, and using the convexity of $f$, we obtain,
after some simple algebra, 
\begin{equation*}
\begin{array}{rcl}
f(z)-f(\hat{x}) & \geq & \langle z-\hat{x},\hat{x^{\ast }}\rangle -\epsilon ,%
\end{array}%
\end{equation*}%
and hence $\hat{x^{\ast }}\in \breve{\partial}f(\epsilon ,\tilde{x})$, as
wanted.
\end{proof}


\begin{Rem}
\label{R1}If $\mathop{\rm gph\,}(T)$ is nonempty, the family $\mathbb{E}(T)$
is nonempty and its biggest enlargement is $T^{BE}$, defined in \eqref{teps}%
. Using this fact, one can easily prove that for every $E\in \mathbb{E}(T)$
and $x\in X,$\ one has $E(0,x)=T(x).$\ Moreover, from Lemma \ref{TF} and the
definitions, it follows that the enlargement $\breve{\partial}f\in \mathbb{E}%
(\partial f)$ (see also \cite{BS02}).
\end{Rem}

\subsection{Convex representations of $T$}

As a consequence of \eqref{eq:9} and \eqref{eq:10}, the function $f^{FY}$ is
an example of a convex function that completely characterizes the graph of
the operator $\partial f$, as well as the graph of the Brøndsted-Rockafellar
enlargement. For an arbitrary maximally monotone operator $T$, Fitzpatrick
defined in \cite[Definition 3.1]{F} an ingenious proper convex $(\Vert \cdot
\Vert \times w)$-lower semicontinuous function, here denoted by $\mathcal{F}_{T}$, that
has the same properties: {%
\begin{equation}
\mathcal{F}_{T}(x,x^{\ast }):=\sup \{\langle y,x^{\ast }\rangle +\langle
x-y,y^{\ast }\rangle :(y,y^{\ast })\in \mathop{\rm gph\,}(T)\}.  \label{eq:2}
\end{equation}%
It satisfies:}

\begin{itemize}
\item $\mathcal{F}_{T}(x,x^{\ast })\geq \langle x,x^{\ast }\rangle $.

\item In analogy to \eqref{eq:9}, we have that (see \cite{F}): 
\begin{equation*}
\mathop{\rm gph\,}(T):=\{(x,x^{\ast })\in X\times X^{\ast }:\;\mathcal{F}%
_{T}(x,x^{\ast })=\mathcal{F}_{T}^*(x^{\ast },x)=\langle x,x^{\ast }\rangle
\}.
\end{equation*}

\item In analogy to \eqref{eq:10}, we have that (see \cite{BS02}):%
\begin{equation*}
x^{\ast }\in T^{\mathrm{BE}}(\epsilon ,x)\hbox{ if and only if }\mathcal{F}%
_{T}(x,x^{\ast })\leq \langle x,x^{\ast }\rangle +\epsilon .
\end{equation*}%
Therefore $\mathcal{F}_{T}$ completely characterizes the graph of the
operator $T$, as well as the graph of its enlargement $T^{\mathrm{BE}}$.
When $T=\partial f$, we can relate $\mathcal{F}_{T}$ and $f^{FY}$ as
follows. 
\begin{equation*}
\forall (x,x^{\ast })\in X\times X^{\ast },\quad \langle x,x^{\ast }\rangle
\leq {\mathcal{F}}_{\partial f}(x,x^{\ast })\leq f(x)+f^{\ast }(x^{\ast
})=f^{FY}(x,x^{\ast }).
\end{equation*}
\end{itemize}

\begin{Rem}
\label{alicante}Note that the Fitzpatrick function associated to a
subdifferential operator could be different from the Fenchel-Young function.
Indeed, if $X$ is a Hilbert space and $f:X\longrightarrow \mathbb{R}$ is
given by $f\left( x\right) :=\frac{1}{2}\left\Vert x\right\Vert ^{2},$ then $%
f^{FY}\left( x,y\right) :=\frac{1}{2}\left( \left\Vert x\right\Vert
^{2}+\left\Vert y\right\Vert ^{2}\right) $ and $\mathcal{F}_{\partial
f}\left( x,y\right) =\frac{1}{4}\left\Vert x+y\right\Vert ^{2}$.
\end{Rem}

The Fitzpatrick function was unnoticed for several years until it was
rediscovered by Martínez-Legaz and Théra \cite{ML01}. However, we recently
discovered, by reading a paper by Fl{å}m \cite{flam}, that this function had
already been used by Krylov \cite{krylov} before Fitzpatrick. According to
the fact that it bridges monotone operators to convex functions, it has been
the subject of an intense research with applications in different areas such
as the variational representation of (nonlinear) evolutionary PDEs, and the
development of variational techniques for the analysis of their structural
stability; see e.g., \cite%
{RocheRossiStefanelli2014,GHOUS08,Visintin20010,Visintin20013}; more
surprisingly, Fl{å}m \cite{flam} gave an economic interpretation of the
Fitzpatrick function.

Moreover, in \cite{F} Fitzpatrick also defined a family of convex functions
associated to $T$. We recall this definition next.

\begin{Def}
\label{def:h} Let $T:X\rightrightarrows X^{\ast }$ be maximally monotone.
Define $\mathcal{H}(T)$ as the family of lower semicontinuous \ convex functions $h:X\times
X^{\ast }\longrightarrow \mathbb{R\cup }\left\{ +\infty \right\} $ such that 
\begin{eqnarray}
&&h(x,x^{\ast })\geq \langle x^{\ast },x\rangle ,\forall x\in X,x^{\ast }\in
X^{\ast },  \label{eq:defh1} \\
&&x^{\ast }\in T(x)\Rightarrow h(x,x^{\ast })=\langle x^{\ast },x\rangle .
\label{eq:defh2}
\end{eqnarray}
\end{Def}

The family $\mathcal{H}(T)$ was studied in \cite{F} in connection with the
operator $T$ itself. It was proved in \cite{F} that the smallest element of
this family is precisely $\mathcal{F}_{T}$.

Clearly, relations (\ref{eq:defh1}) and (\ref{eq:defh2}) imply that one can
express a monotone relation as a minimization problem: setting $\Theta
(x,x^{\ast }):={\mathcal{F}}_{T}(x,x^{\ast })-\langle x,x^{\ast }\rangle ,$
we have that 
\begin{equation*}
x^{\ast }\in T(x)\;\iff \;\Theta (x,x^{\ast })=\inf_{(y,y^{\ast })\in
X\times X^{\ast }}\Theta (y,y^{\ast })=0.
\end{equation*}%
Moreover, it can be observed that, for a prescribed $x^{\ast }$ in the range
ot $T$, i.e. a point $x^*\in T(x)$ for some $x$, one can solve the inclusion 
$x^{\ast }\in T(x)$ just by minimizing the functional $\Theta (\cdot
,x^{\ast })$.

In the paper \cite{Visintin2014}, Visintin presents an interesting
application of Fitzpatrick functions to the Calculus of Variations. As
pointed out above, one can express a monotone relation as a minimization
problem in which the minimum value is prescribed as zero. In \cite%
{Visintin2014} it is shown that, by generalizing the Fitzpatrick approach,
one can express a monotone relation as a minimization problem, without the
need of prescribing the minimum value as zero. This is convenient in many
practical problems in which the minimum value is not known, including
problems from the Calculus of Variations.

Given a maximally monotone operator, \cite[Corollary 3.7]{BS02} shows that
the converse of \eqref{eq:defh2} also holds. Namely, for all $h\in \mathcal{H%
}(T)$ one has 
\begin{equation*}
h(x,x^{\ast })=\langle x,x^{\ast }\rangle \;\iff \;(x,x^{\ast })\in %
\mathop{\rm gph\,}(T).
\end{equation*}%
The use of Fitzpatrick functions has led to considerable simplifications in
the proofs of some classical properties of maximally monotone operators;
see, for instance, the work by Burachik and Svaiter \cite{BS02}, Simons and Z%
{\u{a}}linescu \cite{SZ}, Penot and Z{\u{a}}linescu \cite{PZ}, Bo{\c{t}} et
al. \cite{bot07}, Simons \cite{simons08}, and Marques Alves and Svaiter \cite%
{marques}. It was proved by Burachik and Svaiter \cite{BS02} that the family 
$\mathcal{H}(T)$ is in a one-to-one relationship with the family $\mathbb{E}%
(T)$ of enlargements of $T$, introduced and studied by Svaiter in \cite{SV}.
More connections between $\mathbb{E}(T)$ and $\mathcal{H}(T)$ were studied
in \cite{BS99,BS02,BS06}. The correspondence from $T$ to $\mathcal{H}(T)$
associates to a given maximally monotone operator, functions defined in $%
X\times X^{\ast }$. In the paper \cite{F}, Fitzpatrick also defined a
correspondence which goes in the opposite direction. Namely, given a proper
convex function $h:X\times X^{\ast }\rightarrow \mathbb{R}\cup \{+\infty \}$%
, Fitzpatrick defined the operator $T_{h}:X\rightrightarrows X^{\ast },$
given by 
\begin{equation}
T_{h}(x):=\{x^{\ast }:(x^{\ast },x)\in \partial h(x,x^{\ast })\}  \label{Tg}
\end{equation}

\begin{Rem}
\label{R0}Let $f:X\times X^{\ast }\longrightarrow \mathbb{R\cup }\left\{
+\infty \right\} $ be convex and lower semicontinuous, and consider again $%
f^{FY}(x,x^{\ast })=f(x)+f^{\ast }(x^{\ast })$. Then Example 2.3 in \cite{F}
proves that $T_{f^{FY}}=\partial f$. We extend this result in Theorem \ref%
{Theo1}(ii). Namely, we will extend this equality between two maximally
monotone operators to an equality between two enlargements of $T=\partial f$.
\end{Rem}

The following theorem summarizes those results in \cite{F} which will be
relevant to our study.

\begin{Thm}
\cite{F} \label{Theo0} Let $T:X\rightrightarrows X^{\ast }$ be monotone and $%
f:X\times X^{\ast }\longrightarrow \mathbb{R\cup }\left\{ +\infty \right\} $
be convex. Let $T_{\mathcal{F}_{T}}$ be defined as in \eqref{Tg} for $h:=%
\mathcal{F}_{T}$. The following facts hold.

\begin{itemize}
\item[(a)] For any $x\in X$ one has $T(x)\subset T_{\mathcal{F}_{T}}(x).$ If
T is maximally monotone then $T=T_{\mathcal{F}_{T}}$;

\item[(b)] If $T$ is maximally monotone, then $\mathcal{F}_{T}\in \mathcal{H}%
(T)$. Moreover, $\mathcal{F}_{T}$ is the smallest convex function in $%
\mathcal{H}(T)$;

\item[(c)] The operator $T_{f}$ is monotone.
\end{itemize}
\end{Thm}

We end this subsection by extending Theorem \ref{Theo0}(a) to every $h\in 
\mathcal{H}(T)$; the result is an easy consequence of \cite[Theorem 2.4 and
Proposition 2.2]{F}.

\begin{Pro}
\label{pro2:Th} Let $T$ be maximally monotone, and fix $h\in \mathcal{H}(T)$%
. Then $T=T_{h}$.
\end{Pro}

\begin{proof}
Since $h\in \mathcal{H}(T),$ we have $\mathop{\rm gph\,}T\subset \left\{
(x,x^{\ast })\in X\times X^{\ast }:h(x,x^{\ast })=\langle x^{\ast },x\rangle
\right\} ;$ hence, by \cite[Theorem 2.4]{F}, the inclusion $%
\mathop{\rm
gph\,}T\subset \mathop{\rm gph\,}T_{h}$ holds. On the other hand, by \cite[%
Proposition 2.2]{F}, $T_{h}$ is monotone. Using this fact together with the
preceding inclusion and the maximal monotonicity of $T,$ we get $T=T_{h}$.
\end{proof}

%

\subsection{Autoconjugate convex representations of $T$}

Every element $h\in \mathcal{H}(T)$ is defined on $X\times X^{\ast },$ while 
$h^{\ast }$ is defined on $X^{\ast }\times X$. Recall that the dual of $%
X\times X^{\ast }$ can be identified with $X^{\ast }\times X$ through the
product 
\begin{equation*}
\langle (x,x^{\ast }),(y^{\ast },y)\rangle =\langle x,x^{\ast }\rangle
+\langle y,x^{\ast }\rangle .
\end{equation*}%
In order to work with functions defined on $X\times X^{\ast }$, we will use
the permutation function $i:X\times X^{\ast }\longrightarrow X^{\ast }\times
X$ defined by $i\left( x,x^{\ast }\right) :=\left( x^{\ast },x\right) .$ The
composition $h^{\ast }\circ i$ will thus be defined on $X\times X^{\ast }.$
Notice that the mapping $h\longmapsto h^{\ast }\circ i$ is precisely the
operator $\mathcal{J}:\mathcal{H}(T)\longrightarrow \mathcal{H}(T)$ defined
in \cite{BS02}, which, as shown in \cite[Remark 5.4]{BS02}, is an involution%
\textrm{.}

\begin{Def}
\label{def:conv-repr-t} Let $T:X\rightrightarrows X^{\ast }$ be maximally
monotone. Every $h\in \mathcal{H}(T)$ is called a \emph{convex
representation of }$T$. When $h\in \mathcal{H}(T)$ satisfies 
\begin{equation*}
h^{\ast }\circ i=h,
\end{equation*}%
we say that $h$ is an \emph{autoconjugate convex representation of }$T$.
\end{Def}

The function $f^{FY}$, which characterizes the epsilon-subdifferential
enlargement of $\partial f$, is an autoconjugate convex representation of $%
\partial f$, as can be easily checked. Hence, it is natural to look for
autoconjugate convex representations when searching for an enlargement
structurally closer to the epsilon-subdifferential. This observation
generates a great interest in constructing autoconjugate convex
representations of an arbitrary operator $T$. Outside the subdifferential
case, the operator $T:=\partial f+S,$ where $S$ is a skew-adjoint linear
operator  $(S^*= - S)$, admits the autoconjugate convex representation given by $f\left(
x\right) +f^{\ast }\left( -S\left( x\right) +x^{\ast }\right)$ (see for instance Example 2.6 in \cite{BWY10}, and Ghoussoub \cite{ghoussoub06}). The
interest of having autoconjugate convex representations is also given by the
next theorem:

\begin{Thm}
An operator $T:X\rightrightarrows X^{\ast }$ is maximally monotone if and
only if it admits an autoconjugate convex representation.
\end{Thm}

\begin{proof}
  Svaiter proved in \cite[Proposition 2.2 and Theorem 2.4]{S03} that
  for every maximally monotone operator $T$, there exists $h\in {\cal
    H}(T)$ such that $h$ is auto conjugate. See also Bauschke and Wang \cite[ Theorem
  5.7]{bauschke-wang}. The converse
  follows from a result by Burachik and Svaiter \cite[Theorem
  3.1]{regina-benar}. 
\end{proof}

\begin{Rem}
The ``only if'' part of Theorem 2.2 proved in \cite[Proposition 2.2 and Theorem 2.4]{S03}  is valid in any real Banach
space. The ``if'' part proved in \cite[Theorem
  3.1]{regina-benar} assumes $X$ is reflexive.
\end{Rem}

\begin{Rem}
The papers \cite{penot03,penot,S03} present non-constructive examples of
autoconjugate convex representations of $T$. Constructive examples of
autoconjugate convex representations of $T$ can be found in
\cite{BWY10,bauschke-wang,PZ}.
The one found in \cite{PZ} requires a mild constraint qualification, namely,
that the affine hull of the domain of $T$ is closed. The other ones do not
require any constraint qualification. We will show later other constructive
examples of autoconjugate representations of $T$.
\end{Rem}

We introduce now another map, defined on the set%
\begin{equation*}
\mathcal{H}:=\bigcup \left\{ \mathcal{H}(T):T:X\rightrightarrows X^{\ast }%
\text{ is maximally monotone}\right\} ,
\end{equation*}%
which will have an important role in the definition of our enlargements and
in obtaining autoconjugate convex representations of $T$.

\begin{Rem}
\label{FY auto}For every lower semicontinuous proper convex function $f:X\longrightarrow 
\mathbb{R\cup }\left\{ +\infty \right\} ,$ one has $f^{FY}\in \mathcal{H}%
\left( \partial f\right) \subset \mathcal{H}.$ Furthermore, it is easy to
see that $f^{FY}$ is an autoconjugate convex representation of $\partial f.$
\end{Rem}

\begin{Def}
\label{def:A}The map $\mathcal{A}$$:\mathcal{H}\rightarrow \mathcal{H}$ is
defined by%
\begin{equation}
{\mathcal{A}}h:=\frac{1}{2}\left( h+h^{\ast }\circ i\right) .  \label{eq:61}
\end{equation}
\end{Def}

\begin{Rem}
\label{rem:BS}It follows from \cite[Theorem 5.1 and Proposition 5.3]{BS02}
that $\mathcal{A}h$ and $h^{\ast }\circ i$ belong to $\mathcal{H}(T),$ for
every $h\in \mathcal{H}(T);$ therefore, the map $\mathcal{A}$ is well
defined.
\end{Rem}

For the next theorem we need to define the following sets: 
\begin{equation}
\begin{array}{l}
\\ 
\mathcal{H}_{\ast \leq }:=\{h\in \mathcal{H}:h^{\ast }\circ i\leq h\},\quad {%
\mathcal{H}}_{\ast =}:=\{h\in \mathcal{H}:{h}^{\ast }\circ i=h\}, \\ 
\\ 
\quad \quad \quad \quad {\ \mathcal{H}}_{\ast \geq }:=\{h\in \mathcal{H}%
:h^{\ast }\circ i\geq h\}.%
\end{array}
\label{eq:13}
\end{equation}

\begin{Thm}
\label{theo:JE}Consider the operator $\mathcal{A}$ given in Definition \ref%
{def:A}, and the sets defined in \eqref{eq:13}. The following statements
hold.

\begin{itemize}
\item[(i)] The operator $\mathcal{A}$ maps $\mathcal{H}$ into $\mathcal{H}%
_{\ast \leq }$. The operator $h\longmapsto \left( \mathcal{A}h\right) ^{\ast
}\circ i$ maps $\mathcal{H}$ into $\mathcal{H}_{\ast \geq }$;

\item[(ii)] The set of fixed points of $\mathcal{A}$ is $\left\{ h\in 
\mathcal{H}_{\ast \leq }:h^{\ast }\circ i=h\text{ on }\mathrm{dom}\left(
h\right) \right\};$

\item[(iii)] Let $h\in \mathcal{H}.$ For every $n\geq 1,$ one has $\mathrm{%
dom}\left( \mathcal{A}^{n}{h}\right) =\mathrm{dom}\left( {h}\right) \cap 
\mathrm{dom}\left( h^{\ast }\circ i\right);$

\item[(iv)] Let $h\in \mathcal{H}.$ The sequences $\{{\mathcal{A}}^{n}{h}%
\}_{n\geq 1}\subset \mathcal{H}_{\ast \leq }$ and $\{\left( \mathcal{A}%
^{n}h\right) ^{\ast }\circ i\}_{n\geq 1}\subset \mathcal{H}_{\ast \geq }$
are pointwise non-increasing and non-decreasing, respectively. The pointwise
limit $\mathcal{A}$$^{\infty }{h}$ of the first one satisfies $\mathrm{dom}%
\left( \mathcal{A}^{\infty }{h}\right) =\mathrm{dom}\left( {h}\right) \cap 
\mathrm{dom}\left( h^{\ast }\circ i\right) .$ If $\mathcal{A}$$^{\infty }{h}$
is lower semicontinuous, it is a fixed point of $\mathcal{A};$

\item[(v)] Let $h\in \mathcal{H}.$ For every $n\geq 1$, one has%
\begin{equation}
\left( \mathcal{A}^{n}h\right) ^{\ast }\circ i\leq \mathcal{A}^{\infty }{h}%
\leq \mathcal{A}^{n}{h};  \label{eq:16}
\end{equation}

\item[(vi)] Let $h\in \mathcal{H}.$ The sequence $\{\left( \mathcal{A}%
^{n}h\right) ^{\ast }\circ i\}_{n\geq 1}$ converges pointwise to $\mathcal{A}
$$^{\infty }{h}$ on $\mathrm{dom}\left( {h}\right) \cap \mathrm{dom}\left(
h^{\ast }\circ i\right) ;$

\item[(vii)] Let $T:X\rightrightarrows X^{\ast }$ be maximally monotone$.$
If $h\in \mathcal{H}(T)$ and $\mathcal{A}$$^{\infty }{h}$ is lower semicontinuous,  then $%
\mathcal{A}$$^{\infty }{h}\in \mathcal{H}(T).$
\end{itemize}
\end{Thm}

\begin{proof}
(i) We need to show that $\left( \mathcal{A}h\right) ^{\ast }\circ i\leq $$%
\mathcal{A}$${h}$. Indeed, using the properties of the conjugation operator,
we can write%
\begin{eqnarray*}
\left( \left( \mathcal{A}h\right) ^{\ast }\circ i\right) (x,x^{\ast })
&=&\left( {\mathcal{A}{h}}\right) ^{\ast }(x^{\ast },x) \\
&=&(\frac{h+h^{\ast }\circ i}{2})^{\ast }(x^{\ast },x)=\frac{1}{2}(h+h^{\ast
}\circ i)^{\ast }(2\,x^{\ast },2\,x) \\
&\leq &\frac{1}{2}(h^{\ast }\oplus \left( h^{\ast }\circ i\right) ^{\ast
})(2\,x^{\ast },2\,x) \\
&\leq &\frac{1}{2}\left( h^{\ast }(x^{\ast },x)+\left( h^{\ast }\circ
i\right) ^{\ast }(x^{\ast },x)\right) \\
&=&\frac{1}{2}\left( \left( h^{\ast }\circ i\right) (x,x^{\ast })+{h}%
(x,x^{\ast })\right) =\mathcal{A}{h}(x,x^{\ast }),
\end{eqnarray*}%
where we have used \eqref{eq:12} in the first inequality, the definition of
infimal convolution in the second inequality, and the equality $h^{\ast \ast
}=h$ in the last step (recall that $h$ is lower semicontinuous, convex and proper). The
fact that the operator $h\longmapsto \left( \mathcal{A}h\right) ^{\ast
}\circ i$ maps $\mathcal{H}$ into $\mathcal{H}_{\ast \geq }$ follows from
(i) and the fact that {$\mathcal{A}$}${h}^{\ast \ast }=\mathcal{A}{h}\geq
\left( \mathcal{A}h\right) ^{\ast }\circ i$.

(ii) If the equality $h^{\ast }\circ i=h$ holds on $\mathrm{dom}\left( {h}%
\right) ,$ then we clearly have $\mathcal{A}$${h}=h$ at points where $h$ is
finite. At points where $h$ is infinite, $\mathcal{A}h$ must also be
infinite (because $h^{\ast }$ is proper), and hence $\mathcal{A}$${h}=h$
everywhere. Conversely, assume that $\mathcal{A}$${h}=h$. If $(x,x^{\ast
})\in \mathrm{dom}\left( {h}\right) $, then the equality $\mathcal{A}{h}=h$
yields $\left( h^{\ast }\circ i\right) (x,x^{\ast })<+\infty $ and hence $%
\left( h^{\ast }\circ i\right) (x,x^{\ast })=h(x,x^{\ast })$. This implies
that $h^{\ast }\circ i=h\text{ on }\mathrm{dom}\left( h\right) $.

(iii) We prove the claim by induction. It is clear from the definition that%
\newline
$\mathrm{dom}\left( \mathcal{A}{h}\right) =\mathrm{dom}\left( {h}\right)
\cap \mathrm{dom}\left( h^{\ast }\circ i\right) ,$ so the claim is true for $%
n=1$. Assume that $\mathrm{dom}\left( \mathcal{A}^{n}{h}\right) =\mathrm{dom}%
\left( {h}\right) \cap \mathrm{dom}\left( h^{\ast }\circ i\right) $. Using
(i) yields 
\begin{equation}
\left( {\mathcal{A}^{n}}h\right) ^{\ast }\circ i\leq \mathcal{A}^{n}{h},
\label{eq:14}
\end{equation}%
for every $n\geq 2$. This implies that 
\begin{equation*}
\mathrm{dom}\left( \mathcal{A}^{n}{h}\right) \subset \mathrm{dom}\left(
\left( {\mathcal{A}^{n}}h\right) ^{\ast }\circ i\right) .
\end{equation*}%
Using the definition of ${\mathcal{A}}$, the inclusion above, and the
induction hypothesis, we can write%
\begin{eqnarray*}
&&\mathrm{dom}\left( \mathcal{A}^{n+1}{h}\right) \\
&=&\mathrm{dom}\left( \mathcal{A}^{n}{h}\right) \cap \mathrm{dom}\left(
\left( {\mathcal{A}^{n}}h\right) ^{\ast }\circ i\right) \\
&=&\mathrm{dom}\left( \mathcal{A}^{n}{h}\right) =\mathrm{dom}\left( {h}%
\right) \cap \mathrm{dom}\left( h^{\ast }\circ i\right) ,
\end{eqnarray*}%
%
%
%
%
which proves the claim by induction.

(iv) By \eqref{eq:14} we can write 
\begin{equation}
\mathcal{A}^{n+1}{h}=\frac{1}{2}\left( \mathcal{A}^{n}{h+}\left( {\mathcal{A}%
^{n}}h\right) ^{\ast }\circ i\right) \leq \mathcal{A}^{n}{h},  \label{eq:15}
\end{equation}%
showing that the sequence $\{\mathcal{A}^{n}{h}\}$ is pointwise
non-increasing. By the order reversing property of the conjugation operator,
the sequence $\left\{ \left( {\mathcal{A}^{n}}h\right) ^{\ast }\circ
i\right\} $ is non-decreasing.

Denote by $D_{0}$ the set $\mathrm{dom}\left( {h}\right) \cap \mathrm{dom}%
\left( h^{\ast }\circ i\right) $. We claim that, for $\left( x,x^{\ast
}\right) \in D_{0},$ the two sequences $\left\{ \mathcal{A}^{n}{h}\left(
x,x^{\ast }\right) \right\} $ and $\left\{ \left( \left( {\mathcal{A}^{n}}%
h\right) ^{\ast }\circ i\right) \left( x,x^{\ast }\right) \right\} $ are
adjacent (that is, $\left\{ \mathcal{A}^{n}{h}\left( x,x^{\ast }\right)
\right\} $ is non-increasing, $\left\{ \left( \left( {\mathcal{A}^{n}}%
h\right) ^{\ast }\circ i\right) \left( x,x^{\ast }\right) \right\} $ is
non-decreasing, and $\lim_{n\rightarrow \infty }\left( \mathcal{A}^{n}{h}%
\left( x,x^{\ast }\right) -\left( \left( {\mathcal{A}^{n}}h\right) ^{\ast
}\circ i\right) \left( x,x^{\ast }\right) \right) =0$). We can write 
\begin{eqnarray*}
0 &\leq &(\mathcal{A}^{n+1}h-\left( \mathcal{A}^{n+1}h\right) ^{\ast }\circ
i)(x,x^{\ast })\leq (\mathcal{A}^{n+1}h-\left( \mathcal{A}^{n}h\right)
^{\ast }\circ i)(x,x^{\ast }) \\
&=&\frac{1}{2}(\mathcal{A}^{n}h-\left( \mathcal{A}^{n}h\right) ^{\ast }\circ
i)(x,x^{\ast })<+\infty ,
\end{eqnarray*}%
where we have used (i) and (iii) in the left-most inequality, (\ref{eq:15})
in the second one, the definition of $\mathcal{A}$ in the equality, and the
fact that $(x,x^{\ast })\in D_{0}$ together with (iii) in the last
inequality. Hence, we obtain%
\begin{equation*}
0\leq (\mathcal{A}^{n}{h}-\left( \mathcal{A}^{n}h\right) ^{\ast }\circ
i)(x,x^{\ast })\leq \frac{1}{2^{n-1}}(\mathcal{A}{h}-\left( \left( \mathcal{A%
}h\right) ^{\ast }\circ i)\right) (x,x^{\ast }),
\end{equation*}%
the second inequality following by induction from the above inequality $(%
\mathcal{A}^{n+1}h-\left( \mathcal{A}^{n+1}h\right) ^{\ast }\circ
i)(x,x^{\ast })\leq \frac{1}{2}(\mathcal{A}^{n}h-\left( \mathcal{A}%
^{n}h\right) ^{\ast }\circ i)(x,x^{\ast }),$ and the claim is established.
By (i), the sequence $\{\mathcal{A}^{n}{h}(x,x^{\ast })\}$ is bounded below
by the function $\pi :=\langle \cdot ,\cdot \rangle $. Therefore, for every
fixed $(x,x^{\ast })\in D_{0}$, the sequence $\{\mathcal{A}^{n}{h}(x,x^{\ast
})\}\subset \mathbb{R}$ is non-increasing and bounded below by $\langle
x,x^{\ast }\rangle \in \mathbb{R}$. The completeness axiom thus yields 
\begin{equation*}
\mathbb{R}\ni \lim_{n\rightarrow \infty }\left( \mathcal{A}^{n}{h}\right)
(x,x^{\ast })=\inf_{n}\left( \mathcal{A}^{n}{h}\right) (x,x^{\ast })\geq
\langle x,x^{\ast }\rangle .
\end{equation*}%
Note that, if $(x,x^{\ast })\not\in D_{0}$, then $\left( \mathcal{A}^{n}{h}%
\right) (x,x^{\ast })=+\infty $ for all $n$, so in this case we have $\left( 
\mathcal{A}^{\infty }h\right) (x,x^{\ast })=+\infty $. From its definition,
we have that $\mathcal{A}^{\infty }h$ is proper and convex, and 
\begin{equation*}
\mathrm{dom}\left( {\mathcal{A}^{\infty }h}\right) =\mathrm{dom}\left( {h}%
\right) \cap \mathrm{dom}\left( h^{\ast }\circ i\right) .
\end{equation*}%
To prove that $\mathcal{A}^{\infty }h$ is a fixed point of $\mathcal{A}$
provided that it is lower semicontinuous we will use (ii). We have just shown that $D_{0}=%
\mathrm{dom}\left( {\mathcal{A}^{\infty }h}\right) $. We need to prove that $%
\left( \mathcal{A}^{\infty }h\right) ^{\ast }\circ i=\mathcal{A}^{\infty }h$
on $D_{0}$. Indeed, take $(x,x^{\ast })\in D_{0}$. By (iii) and \eqref{eq:14}%
, the sequences $\{\left( \mathcal{A}^{n}h\right) (x,x^{\ast })\}$ and $%
\{\left( \left( \mathcal{A}^{n}h\right) ^{\ast }\circ i\right) (x,x^{\ast
})\}$ are contained in $\mathbb{R}$. We have shown that the sequence 
$\{\left( \mathcal{A}^{n}h\right) (x,x^{\ast })\}$ converges monotonically.
Since, as noted earlier, the sequences  
$\{\left( \mathcal{A}^{n}h\right) (x,x^{\ast })\}$ and $\{\left( \left( 
\mathcal{A}^{n}h\right) ^{\ast }\circ i\right) (x,x^{\ast })\}$ are
adjacent, they have the same limit $\left( \mathcal{A}^{\infty }h\right)
(x,x^{\ast })$. Using this fact, for every $(x,x^{\ast })\in D_{0}$ we can
write%
\begin{eqnarray*}
\left( \left( \mathcal{A}^{\infty }h\right) ^{\ast }\circ i\right)
(x,x^{\ast }) &=&\left( \mathcal{A}^{\infty }h\right) ^{\ast }(x^{\ast },x)
\\
&=&\left( \inf_{n}\mathcal{A}^{n}{h}\right) ^{\ast }(x^{\ast
},x)=\sup_{n}\left( \mathcal{A}^{n}{h}\right) ^{\ast }(x^{\ast },x) \\
&=&\sup_{n}\left( \left( {\mathcal{A}^{n}}h\right) ^{\ast }\circ i\right)
(x,x^{\ast })=\lim_{n\rightarrow \infty }\left( \left( {\mathcal{A}^{n}}%
h\right) ^{\ast }\circ i\right) (x,x^{\ast }) \\
&=&\lim_{n\rightarrow \infty }\left( \mathcal{A}^{n}{h}\right) (x,x^{\ast
})=\left( \mathcal{A}^{\infty }h\right) (x,x^{\ast }),
\end{eqnarray*}%
showing that $\left( \mathcal{A}^{\infty }h\right) ^{\ast }\circ i=$$%
\mathcal{A}^{\infty }h$ on $D_{0}$. For $(x,x^{\ast })\notin D_{0}=\mathrm{%
dom}\left( {\mathcal{A}^{\infty }h}\right) $, we trivially have%
\begin{equation*}
\left( \left( \mathcal{A}^{\infty }h\right) ^{\ast }\circ i\right)
(x,x^{\ast })\leq \left( \mathcal{A}^{\infty }h\right) (x,x^{\ast })=+\infty
=\lim_{n\rightarrow \infty }\left( \mathcal{A}^{n}{h}\right) (x,x^{\ast }),
\end{equation*}%
the latter equality following from (iii). Hence $\mathcal{A}^{\infty }h\in 
\mathcal{H}_{\ast \leq }$. This and (ii) prove that $\mathcal{A}^{\infty }h$
is a fixed point of $\mathcal{A}$.

(v) The second inequality in \eqref{eq:16} follows from the definition of $%
\mathcal{A}^{\infty }h$. The first inequality follows from the monotonicity
of the sequences $\left\{ {\mathcal{A}^{n}}h\right\} _{n\geq 1}$ and $%
\left\{ \left( {\mathcal{A}^{n}}h\right) ^{\ast }\circ i\right\} _{n\geq 1}$
combined with (\ref{eq:14}). Indeed, for every $n\geq 1$ we have%
\begin{equation*}
\left( {\mathcal{A}^{n}}h\right) ^{\ast }\circ i\leq \sup_{m\geq 1}\left( {%
\mathcal{A}^{m}}h\right) ^{\ast }\circ i\leq \inf_{m\geq 1}{\mathcal{A}^{m}}%
h=\mathcal{A}^{\infty }h.
\end{equation*}

(vi) From (v), for every $(x,x^{\ast })\in \mathrm{dom}\left( {h}\right)
\cap \mathrm{dom}\left( h^{\ast }\circ i\right) $ we have 
\begin{equation*}
\left( \mathcal{A}^{\infty }h\right) (x,x^{\ast })=\lim_{n}\left( \left( {%
\mathcal{A}^{n}}h\right) ^{\ast }\circ i\right) (x,x^{\ast }).
\end{equation*}%
Since, according to the proof of (iv), for $(x,x^{\ast })\in \mathrm{dom}%
\left( {h}\right) \cap \mathrm{dom}\left( h^{\ast }\circ i\right) $ the
sequences $\{\left( \mathcal{A}^{n}h\right) (x,x^{\ast })\}$ and $\{\left(
\left( \mathcal{A}^{n}h\right) ^{\ast }\circ i\right) (x,x^{\ast })\}$ are
adjacent, we immediately obtain that $\{\left( \left( \mathcal{A}%
^{n}h\right) ^{\ast }\circ i\right) (x,x^{\ast })\}$ converges to $\left( 
\mathcal{A}^{\infty }h\right) (x,x^{\ast }).$

(vii) By (v) and (i) we have%
\begin{equation*}
\left( \mathcal{A}^{\infty }h\right) (x,x^{\ast })\geq \left( \left( 
\mathcal{A}h\right) ^{\ast }\circ i\right) (x,x^{\ast })\geq \left\langle
x,x^{\ast }\right\rangle
\end{equation*}%
for every $(x,x^{\ast })\in X\times X^{\ast }.$ Hence, if $x^{\ast }\in
T\left( x\right) ,$ by (v) and Remark \ref{rem:BS}, we have%
\begin{equation*}
\left\langle x,x^{\ast }\right\rangle \leq \left( \mathcal{A}^{\infty
}h\right) (x,x^{\ast })\leq \left( \mathcal{A}h\right) (x,x^{\ast
})=\left\langle x,x^{\ast }\right\rangle ,
\end{equation*}%
which proves that $\mathcal{A}^{\infty }h\in \mathcal{H}(T)$ provided that $%
\mathcal{A}^{\infty }h$ is lower semicontinuous.%
\end{proof}

\begin{Rem}
Since, according to the proof of (iv), one has $\left( \left( \mathcal{A}%
^{\infty }h\right) ^{\ast }\circ i\right) (x,x^{\ast })=\left( \mathcal{A}%
^{\infty }h\right) (x,x^{\ast })$ for every $(x,x^{\ast })\in \mathrm{dom}%
\left( {\mathcal{A}^{\infty }h}\right) ,$ the function ${\mathcal{A}^{\infty
}h}$ is lower semicontinuous  on its domain. Therefore, for the lower semicontinuity assumption of (iv)
and (vii) to hold, it is sufficient that ${\mathcal{A}^{\infty }h}$ be
lower semicontinuous on the boundary of its domain. In particular, this condition
automatically holds if the set $\mathrm{dom}\left( {h}\right) \cap \mathrm{%
dom}\left( h^{\ast }\circ i\right) $ is closed.
\end{Rem}

\begin{Rem}
\label{FY fixed}Since, by Remark \ref{FY auto}, the Fenchel-Young function $%
f^{FY}$ associated with a  lower semicontinuous proper convex function $f:X\longrightarrow 
\mathbb{R\cup }\left\{ +\infty \right\} $ is an autoconjugate
representation, it is a fixed point of $\mathcal{A}.$
\end{Rem}

We note that the function $\mathcal{A}^{\infty }h$ may fail to be an
autoconjugate convex representation of $T$, as the following example shows.

\begin{Exa}
\label{not auto}Let $T$ be the identity in a Hilbert space. In this case $%
\mathcal{F}_{T}(x,x^{\ast })=\frac{\Vert x+x^{\ast }\Vert ^{2}}{4}$ and%
\begin{equation*}
\left( \mathcal{F}_{T}^{\ast }\circ i\right) (x,x^{\ast })=\left\{ 
\begin{array}{lr}
\Vert x\Vert ^{2}, & \hbox{ if }x=x^{\ast }, \\ 
+\infty , & \hbox{ if }x\not=x^{\ast }.%
\end{array}%
\right.
\end{equation*}%
If we take $h:=\mathcal{F}{_{T}^{\ast }}\circ i$\textrm{\ }then\ it\ is\
easy\ to\ check\ that\ $\mathcal{A}h=h$,\ and\ hence\ $\mathcal{A}^{\infty
}h=h.$ On the other hand, $\mathcal{A}^{\infty }h=h$ is not an autoconjugate
, since\ $h^{\ast }\circ i=\mathcal{F}_{T}$. We have, however, $h^{\ast
}\circ i=h$ on the diagonal, that is, on $\mathrm{dom}(h).$
\end{Exa}

The preceding example shows that, in general, $\mathcal{A}^{\infty }{h}$ may
fail to be an autoconjugate of $T$. The next result establishes an
assumption on $h$ under which $\mathcal{A}^{\infty }{h}$ is an autoconjugate
of $T$.

\begin{Cor}
\label{cor:JE}With the notation of Theorem \ref{theo:JE}, let $h\in \mathcal{%
H}_{\ast \leq }$. Assume that the following qualification condition holds: 
\begin{equation*}
(\mathcal{Q}\mathcal{C})\quad \quad \mathrm{dom}(h)=\mathrm{dom}(h^{\ast
}\circ i) .
\end{equation*}
Then 
\begin{equation*}
\left( \mathcal{A}^{\infty }h\right) ^{\ast }\circ i=\mathcal{A}^{\infty }h.
\end{equation*}
\end{Cor}

\begin{proof}
By Theorem 2.3 parts (iv) and (ii), we have $\left( \mathcal{A}^{\infty }h\right) ^{\ast
}\circ i=\mathcal{A}^{\infty }h$ on $\mathrm{dom}(h).$ On the other hand, by
(v) and (i) of Theorem 2.3, we have $\mathcal{A}^{\infty }{h\leq h,}$ and hence $\left( 
\mathcal{A}^{\infty }{h}\right) ^{\ast }{\geq h}^{\ast }.$ Therefore%
\begin{equation*}
\mathrm{dom}(\left( \mathcal{A}^{\infty }h\right) ^{\ast }\circ i)\subset 
\mathrm{dom}(h^{\ast }\circ i)=\mathrm{dom}(h).
\end{equation*}%
Since, by Theorem 2.3(iv), we have $\mathrm{dom}(\mathcal{A}^{\infty }h)=\mathrm{dom}%
(h),$ it follows that $\mathcal{A}^{\infty }h=+\infty =\left( \mathcal{A}%
^{\infty }h\right) ^{\ast }\circ i$ on $\left( X\times X^{\ast }\right)
\setminus \mathrm{dom}(h).$
\end{proof}

\bigskip

The limit $\mathcal{A}^{\infty }h$ found in the previous result provides a
constructive example of autoconjugate convex representation, in the
following sense.

\begin{Thm}
Consider the operator $\mathcal{A}$ given in Definition \ref{def:A}. Let$%
\mathcal{A}^{\infty }h$ be as in Theorem \ref{theo:JE} and $\left( x,x^{\ast
}\right) \in X\times X^{\ast }.$ If $\max \left\{ \left( \mathcal{A}h\right)
\left( x,x^{\ast }\right) ,\left( \left( \mathcal{A}h\right) ^{\ast }\circ
i\right) \left( x,x^{\ast }\right) \right\} =+\infty ,$ then $\left( 
\mathcal{A}^{\infty }h\right) \left( x,x^{\ast }\right) =+\infty .$ If $\max
\left\{ \left( \mathcal{A}h\right) \left( x,x^{\ast }\right) ,\left( 
\mathcal{A}h\right) ^{\ast }(x^{\ast },x)\right\} <+\infty $ and $\epsilon
>0 $, setting $n>1+\log _{2}\left( \left( \mathcal{A}h\right) \left(
x,x^{\ast }\right) -\left( \left( \mathcal{A}h\right) ^{\ast }\circ i\right)
\left( x,x^{\ast }\right) \right) -\log _{2}\epsilon ,$ one has $\left( 
\mathcal{A}^{n}h\right) (x,x^{\ast })-\epsilon \leq \left( \mathcal{A}%
^{\infty }h\right) \left( x,x^{\ast }\right) ;$ this provides a convenient
stopping criterion for effectively computing $\left( \mathcal{A}^{\infty
}h\right) \left( x,x^{\ast }\right) $ with an error smaller than $\epsilon $
by means of the iteration $\left( \mathcal{A}^{n}h\right) (x,x^{\ast }).$
\end{Thm}

\begin{proof}
If $\left( \mathcal{A}h\right) ^{\ast }(x^{\ast },x)=+\infty ,$ then $\left( 
\mathcal{A}^{\infty }h\right) (x,x^{\ast })=+\infty .$ Let us then assume
that $\left( \mathcal{A}h\right) ^{\ast }(x^{\ast },x)<+\infty .$ If $\left( 
\mathcal{A}h\right) (x,x^{\ast })=+\infty ,$ then, by (iv) and (iii) of
Theorem \ref{theo:JE}, we also have $\left( \mathcal{A}^{\infty }h\right)
(x,x^{\ast })=+\infty .$ Consider now the case when both $\left( \mathcal{A}%
h\right) \left( x,x^{\ast }\right) $ and $\left( \mathcal{A}h\right) ^{\ast
}(x^{\ast },x)$ are finite. For $n\geq 1,$ from the inequalities%
\begin{equation*}
\left( \mathcal{A}h\right) ^{\ast }\circ i\leq \left( \mathcal{A}%
^{n}h\right) ^{\ast }\circ i\leq \mathcal{A}^{\infty }h\leq \mathcal{A}%
^{n}h\leq \mathcal{A}h
\end{equation*}%
and the fact that the sequence $\left\{ \left( \mathcal{A}^{n}h\right)
  ^{\ast }\right\}$ is increasing, it follows that%
\begin{eqnarray*}
0 &\leq &  \left( \mathcal{A}^{n+1}h\right) (x,x^{\ast })-\left( \mathcal{A}^{\infty }h\right) (x,x^{\ast }) \\
&\leq &\left( \mathcal{A}^{n+1}h\right) (x,x^{\ast })-\left( \left( \mathcal{A}^{n+1}h\right) ^{\ast }\circ i\right) (x,x^{\ast }) \\
&\leq &\left( \mathcal{A}^{n+1}h\right) (x,x^{\ast })-\left( \left( \mathcal{A}^{n}h\right) ^{\ast }\circ i\right) (x,x^{\ast }) \\
&=&\frac{1}{2}\left(   \left(\mathcal{A}^{n}h\right) (x,x^{\ast }) + \left(\left(
    \mathcal{A}^{n}h\right) ^{\ast }\circ i\right) (x,x^{\ast })\right)   -
             \left(\left(\mathcal{A}^{n}h\right) ^{\ast }\circ i\right) (x,x^{\ast })\\
&=&\frac{1}{2}\left(   \left(\mathcal{A}^{n}h\right) (x,x^{\ast }) - \left(\left(
    \mathcal{A}^{n}h\right) ^{\ast }\circ i\right) (x,x^{\ast })\right) ;
\end{eqnarray*}
hence%
\begin{equation*}
\left( \mathcal{A}^{n+1}h\right) (x,x^{\ast })-\left( \mathcal{A}^{\infty
}h\right) (x,x^{\ast })\leq \frac{1}{2^{n}}\left( \left( \mathcal{A}h\right)
(x,x^{\ast })-\left( \left( \mathcal{A}h\right) ^{\ast }\circ i\right)
(x,x^{\ast })\right) ;
\end{equation*}%
therefore, if $\left( \mathcal{A}h\right) \left( x,x^{\ast }\right) \neq
\left( \left( \mathcal{A}h\right) ^{\ast }\circ i\right) \left( x,x^{\ast
}\right) $ (otherwise, $\left( \mathcal{A}^{\infty }h\right) \left(
x,x^{\ast }\right) \newline
=\left( \mathcal{A}h\right) \left( x,x^{\ast }\right) $), then for $\epsilon
>0,$ taking $n>1+\log _{2}\left( \left( \mathcal{A}h\right) \left( x,x^{\ast
}\right) -\left( \!\left( \mathcal{A}h\right) ^{\ast }\circ i\right) \left(
x,x^{\ast }\right) \right) -\log _{2}\epsilon ,$ we have%
\begin{equation*}
\frac{1}{2^{n-1}}\left( \left( \mathcal{A}h\right) \left( x,x^{\ast }\right)
-\left( \left( \mathcal{A}h\right) ^{\ast }\circ i\right) \left( x,x^{\ast
}\right) \right) <\epsilon ,
\end{equation*}%
and hence%
\begin{equation*}
\left( \mathcal{A}^{n}h\right) \left( x,x^{\ast }\right) -\epsilon \leq
\left( \mathcal{A}^{\infty }h\right) \left( x,x^{\ast }\right) \leq \left( 
\mathcal{A}^{n}h\right) \left( x,x^{\ast }\right) .
\end{equation*}%

\end{proof}

\section{A family of enlargements}

We now introduce and investigate a family of enlargements, denoted $\mathbb{E%
}_{\mathcal{H}}(T)$, which is inspired by Fitzpatrick's paper \cite{F}. As
we will see next, $\mathbb{E}_{\mathcal{H}}(T)\subset \mathbb{E}(T)$ and
therefore this new family inherits all the good properties of the elements
of $\mathbb{E}(T)$. In its definition we will use the $\epsilon $%
-subdifferential of a function $h\in {\mathcal{H}}(T)$. Equation \eqref{Tg}
gives an operator associated to a convex function $h$ defined on $X\times
X^{\ast }$. We now extend this operator so that it results in a point-to-set
mapping defined on $\mathbb{R}_{+}\times X$.

\begin{Def}
\label{Th:def} Let ${T}:X\rightrightarrows X^{\ast }$ be maximally monotone.
For $h\in {\mathcal{H}}(T)$ and $\epsilon \geq 0,$ we define $\breve{T}_{h}:%
\mathbb{R}_{+}\times X\rightrightarrows X^{\ast }$ by%
\begin{equation}
\breve{T}_{h}(\epsilon ,x):=\{x^{\ast }\in X^{\ast }:(x^{\ast },x)\in \breve{%
{\partial }}h(2\epsilon ,x,x^{\ast })\}.  \label{eq:Th}
\end{equation}
\end{Def}

\begin{Rem}
By Definition \ref{Th:def} and \eqref{Tg}, if $h\in \mathcal{H}(T)$ then 
\begin{equation*}
T_{h}(x)=\breve{T}_{h}(0,x).
\end{equation*}
\end{Rem}

\begin{Rem}
\label{rem:Lh=E} Following Burachik and Svaiter \cite{BS02}, we can associate
with $h\in \mathcal{H}(T)$ an enlargement $L^{h}\in \mathbb{E}(T)$ as
follows. For $\epsilon \geq 0$ and $x\in X$ we set 
\begin{equation*}
L^{h}(\epsilon ,x):=\{x^{\ast }\in X^{\ast }:h\left( x,x^{\ast }\right) \leq
\langle x,x^{\ast }\rangle +\epsilon \}.
\end{equation*}%
Conversely, it was also shown in \cite{BS02} that with every $E\in \mathbb{E}%
(T)$ we can associate a unique $h\in \mathcal{H}(T)$ such that $E=L^{h}$.
\end{Rem}

\begin{Rem}
\label{FY enlargement}For every lower semicontinuous  proper convex function $%
f:X\longrightarrow \mathbb{R\cup }\left\{ +\infty \right\} ,$ one has $%
L^{f^{FY}}=\breve{{\partial }}f.$
\end{Rem}

We next give a specific notation to the unique $h$ associated with an
enlargement $E\in \mathbb{E}(T).$

\begin{Def}
\label{def:hE}\textrm{\ }Given $E\in \mathbb{E}(T)$, denote by $h_{E}$ the
unique $h\in \mathcal{H}(T)$ such that $E=L^{h}$.
\end{Def}

\begin{Pro}
\label{pro:Lh}Let $T:X\rightrightarrows X^{\ast }$ be maximally monotone,
and fix $h\in \mathcal{H}(T)$. Then $\breve{T}_{h}=L^{\mathcal{A}h}\in 
\mathbb{E}(T)$, where $\mathcal{A}$ is as in Definition 2.4.
\end{Pro}

\begin{proof}
Our proof will follow from Remark \ref{rem:Lh=E}. Indeed, from Fenchel-Young
inequality, we see that $x^{\ast }\in \breve{T}_{h}(\epsilon ,x)$ if and
only if 
\begin{equation*}
h(x,x^{\ast })+h^{\ast }(x^{\ast },x)\leq \langle (x,x^{\ast }),(x^{\ast
},x)\rangle +2\epsilon =2\left( \langle x,x^{\ast }\rangle +\epsilon \right)
.
\end{equation*}%
Equivalently, $x^{\ast }\in \breve{T}_{h}(\epsilon ,x)$ if and only if $%
x^{\ast }\in L^{\mathcal{A}h}(\epsilon ,x).$ Hence $\breve{T}_{h}=L^{%
\mathcal{A}h}\in \mathbb{E}(T)$.
\end{proof}


\begin{Cor}
Let $f:X\longrightarrow \mathbb{R\cup }\left\{ +\infty \right\} $ be a lower
semicontinuos proper convex function. Then $\breve{{\partial }}f=\breve{T}%
_{f^{FY}}$.
\end{Cor}

\begin{proof}
By Proposition \ref{pro:Lh} and Remarks \ref{FY fixed} and \ref{FY
enlargement}, we have%
\begin{equation*}
\breve{T}_{f^{FY}}=L^{\mathcal{A}f^{FY}}=L^{f^{FY}}=\breve{{\partial }}f.
\end{equation*}%
%
%
%
%
%
%
%
%
%
%
\end{proof}

\bigskip

Motivated by the preceding result, for a maximally monotone $%
T:X\rightrightarrows X^{\ast }$ we define the following family of
enlargements.%
\begin{equation}
\mathbb{E}_{\mathcal{H}}(T):=\{E\in \mathbb{E}(T):\hbox{there exists }h\in 
\mathcal{H}(T)\hbox{
  s.t. }E=\breve{T}_{h}\}.  \label{eq:8}
\end{equation}

Proposition \ref{pro:Lh} yields the following result.

\begin{Cor}
\label{pro:Th} Let $T$ and $h$ be as in Proposition \ref{pro:Lh}.

\begin{itemize}
\item[(i)] For every $x\in \dom (T)$ and every $\epsilon \geq 0$, the
set $\breve{T}_{h}(\epsilon ,x)$ is convex;

\item[(ii)] The graph of the mapping $\breve{T}_{h}$ is \emph{demi-closed}.
Namely, if $\{x_{n}\}\subset X$ converges strongly (weakly) to $x$, $%
\{x_{n}^{\ast }\}\subset \breve{T}_{h}(\epsilon _{n},x_{n})$ converges
weakly (strongly, respectively) to $x^{\ast }$, and $\{\epsilon _{n}\}$
converges to $\epsilon \geq 0$, then $x^{\ast }\in \breve{T}_{h}(\epsilon
,x) $. In particular, $\breve{T}_{h}(\epsilon ,x)$ is weakly closed;

\item[(iii)] $\breve{T}_{h}(0,x)=T(x)$ for every $x\in X$.
\end{itemize}
\end{Cor}

\begin{proof}
For part (i), we use Proposition \ref{pro:Lh}. Indeed, the function $%
\mathcal{A}h\in \mathcal{H}(T)$ is convex, and hence it is direct to check
that the set $L^{\mathcal{A}h}(\epsilon ,x)=\breve{T}_{h}(\epsilon ,x)$ is
convex. For part (ii), we use again Proposition \ref{pro:Lh} and \cite[%
Proposition 4.3]{SV}. The latter states that every $E\in \mathcal{H}(T)$ has
a demi-closed graph. 
Part (iii) also follows directly from Proposition \ref{pro2:Th}:%
\begin{equation*}
\breve{T}_{h}(0,x)=T_{h}(x)=T\left( x\right).
\end{equation*}
\end{proof}

\bigskip

The following result is another straightforward consequence of Proposition %
\ref{pro:Lh}.

\begin{Cor}
\label{cor:Te} Let $T:X\rightrightarrows X^{\ast }$ be maximally monotone.
Then, for every $h\in \mathcal{H}(T),$ $\epsilon \geq 0$ and $x\in X$, we
have $\breve{T}_{h}(\epsilon ,x)\subset T^{\mathrm{BE}}(\epsilon ,x)$;\ in
particular, $\breve{T}_{\mathcal{F}_{T}}(\epsilon ,x)\subset T^{\mathrm{BE}%
}(\epsilon ,x)$.
\end{Cor}

\begin{proof}
The proof is a direct consequence of Proposition \ref{pro:Lh} and the fact
that $T^{\mathrm{BE}}$ is the biggest element in $\mathcal{H}(T) $. 
\end{proof}

\subsection{\bf Additivity}

Following \cite{SV}, an enlargement $E\in \mathbb{E}(T)$ is said to be \emph{%
\ additive} when for every $x^{\ast }\in E(\epsilon _{1},x)$ and every $%
y^{\ast }\in E(\epsilon _{2},y)$ we have 
\begin{equation}
\langle x-y,x^{\ast }-y^{\ast }\rangle \geq -(\epsilon _{1}+\epsilon _{2}).
\label{eq:3}
\end{equation}%
Given $E\in \mathbb{E}(T)$, and $h_{E}$ as in Definition \ref{def:hE}, we
say that $h_{E}$ is \emph{additive} whenever $E$ is additive. In other
words, $h\in \mathcal{H}(T)$ is additive if and only if $L^{h}$ is additive.
We define the following sets 
\begin{equation*}
\mathcal{H}_{{a}}(T):=\{h\in \mathcal{H}(T):L^{h}\hbox{ is
additive}\},
\end{equation*}%
\begin{equation*}
\mathbb{E}_{{a}}(T):=\{E\in \mathbb{E}(T):E\hbox{ is additive}\}.
\end{equation*}%
In the definition below, we use some sets and notation introduced in \cite%
{S03}.

\begin{Def}
\label{def2:additivity}\textrm{\ }For a maximally monotone $%
T:X\rightrightarrows X^{\ast }$ and $h\in \mathcal{H}(T),$ define 
\begin{equation}
S(h):=\{g\in \mathcal{H}(T):h\geq g\geq g^{\ast }\circ i\}.  \label{eq:5}
\end{equation}%
We say that $g\in S(h)$ is minimal (on $S(h)$) if, whenever there is $%
g^{\prime }\in S(h)$ such that $g^{\prime }\leq g$, we must have $%
g=g^{\prime }$.
\end{Def}

\begin{Rem}
\label{rem1:additivity}Given $E\in \mathbb{E}(T)$, let $h_{E}\in \mathcal{H}%
(T)$ be as in Definition \ref{def:hE}. In other words, we have $E=L^{h_{E}}$%
. Proposition 5.5 in \cite{BS02} states that 
\begin{equation*}
E=L^{h_{E}}\in \mathbb{E}_{{a}}(T)\Longleftrightarrow h_E^{\ast }\circ i\leq
h_{E}.
\end{equation*}
\end{Rem}

\begin{Rem}
\label{rem2:additivity}Definition \ref{def2:additivity} and Remark \ref%
{rem1:additivity} imply that 
\begin{equation*}
E\in \mathbb{E}_{{a}}(T)\Longleftrightarrow S(h_{E})\not=\emptyset
\Longleftrightarrow h_{E}\in S(h_{E}).
\end{equation*}%
Indeed, if $S(h_{E})\not=\emptyset $ then there exists $g\in \mathcal{H} (T)$
such that $h_{E}\geq g\geq g^{\ast }\circ i$. This implies that 
\begin{equation*}
{h_E}^{\ast }\circ i{_{E}}\leq g^{\ast }\circ i\leq g\leq h_{E},
\end{equation*}%
so $h_{E}\in \mathcal{H} _{{a}}(T)$ by Remark \ref{rem1:additivity}. In this
situation, $h_{E}\in S(h_{E})$.
\end{Rem}

It was observed in \cite{SV} that additivity, as a property of the graph,
can be maximal with respect to inclusion. We recall next this maximality
property, an introduce the relation of \emph{\ mutual additivity} between
enlargements as well.

\begin{Def}
Let $T:X\rightrightarrows X^{\ast }$ be maximally monotone.

\begin{itemize}
\item[(a)] We say that $E\in \mathbb{E}_{{a}}(T)$ is \emph{maximally additive%
} (or \emph{max-add}, for short), if, whenever there exists $E^{\prime }\in 
\mathbb{E}_{{a}}(T)$ such that 
\begin{equation*}
E(\epsilon ,x)\subset E^{\prime }(\epsilon ,x),\forall \,\epsilon \geq
0,\,x\in X,
\end{equation*}%
then we must have $E=E^{\prime }$.

\item[(b)] Let $E,E^{\prime }\in \mathbb{E}(T)$. We say that $E$ and $%
E^{\prime }$ are \emph{\ mutually additive}, if for all $\epsilon ,\eta \geq
0,$ $x,y\in X,$ ${x^{\ast }}\in E(\epsilon ,x)$ and ${y^{\ast }}\in
E^{\prime }(\eta ,y)$ we have 
\begin{equation}
\langle x-y,x^{\ast }-y^{\ast }\rangle \geq -(\epsilon +\eta ).  \label{eq:4}
\end{equation}%
We denote this situation as $E\sim _{a}E^{\prime }$.
\end{itemize}
\end{Def}

\begin{Rem}
\label{rem:add1}Note that $E$ is additive if and only if $E\sim _{{a}}E$.
Note also that the relation $\sim _{a}$ is symmetric.
\end{Rem}

\begin{Rem}
\label{rem:add2}Take $h\in \mathcal{H}(T)$ such that $h\geq h^{\ast }\circ i$%
. Theorem 2.4 in \cite{S03} proves that 
\begin{equation}
h_{0}\in S(h)\hbox{ is {minimal} in }S(h)\hbox{ if and only if }h_{0}^{\ast
}\circ i=h_{0}.  \label{eq:7}
\end{equation}%
In other words, minimal elements of $S(h)$ are autoconjugate convex
representations of $T$. We will see that the latter property characterizes
max-add enlargements.
\end{Rem}

\begin{Rem}
For $T=\partial f$, it was proved by Svaiter \cite{SV} that the $\epsilon $%
-subdifferential is max-add. For an arbitrary $T$, it was proved in \cite{SV}
that the smallest enlargement of $T$ is always additive, and the existence
of a max-add enlargement was deduced in \cite{SV} using Zorn's lemma. On the
other hand, additivity does not necessarily hold for $T^{\mathrm{BE}}$, the
biggest enlargement of $T$. More precisely, the following weaker inequality
was established by Burachik and Svaiter, see \cite{BS99}.
\end{Rem}

\begin{Thm}
Let $H$ be a Hilbert space, $T:H\rightrightarrows H$ be maximally monotone,
and $\epsilon ,\eta \geq 0$. Then, 
\begin{equation*}
\langle x-y,x^{\ast }-y^{\ast }\rangle \geq -\left( \sqrt{\epsilon }+\sqrt{%
\eta }\right) ^{2}\quad \forall {x^{\ast }}\in T^{\mathrm{BE}}(\epsilon ,x),{%
y^{\ast }}\in T^{\mathrm{BE}}({\eta },y).
\end{equation*}
\end{Thm}

The following result is independent from the enlargement $\breve{T}_h$ and
is interesting in its own right.

\begin{Pro}
\label{pro:MA} Let $T:X\rightrightarrows X^{\ast }$ be maximally monotone
and $E,E^{\prime }\in \mathbb{E}(T)$. Assume that $h,h^{\prime }\in \mathcal{%
H}(T)$ are such that $E=L^{h}$ and $E^{\prime}=L^{h^\prime }$. The following
hold.

\begin{itemize}
\item[(i)] $E\sim _{a}E^{\prime }$ if and only if $h^{\ast }\circ i\leq
h^{\prime }$. In particular, $E$ is additive if and only if $h^{\ast }\circ
i\leq h$;

\item[(ii)] $h=h^{\ast }\circ i\,\,$ if and only if $L^{h}$ is max-add.
\end{itemize}

In particular, $E$ is mutually additive with $L^{h^{\ast }\circ i}$.
Inasmuch $L^{h^{\ast }\circ i}$ is the largest enlargement which is mutually
additive with $L^{h}$, it can thus be seen as the \textquotedblleft additive
complement\textquotedblright\ of $L^{h}$. Moreover, max-add enlargements are
characterized by the fact that they coincide with their additive complement.
\end{Pro}

\begin{proof}
(i) Assume that (\ref{eq:4}) holds for all ${x^{\ast }}\in E(\epsilon
,x)=L^{h}(\epsilon ,x)$ and all $y^{\ast }\in E^{\prime }(\eta
,y)=L^{h^{\prime }}(\eta ,y)$. For every $(x,{x^{\ast }})\in \mathrm{dom}%
\left( h\right) ,$ set $\epsilon :=h(x,{x^{\ast }})-\left\langle x,x^{\ast
}\right\rangle \geq 0$. Similarly, for $(y,y^{\ast })\in \mathrm{dom}\left(
h^{\prime }\right) ,$ set $\eta :=h^{\prime }(y,y^{\ast })-\left\langle
y,y^{\ast }\right\rangle \geq 0$. Using \eqref{eq:4}, we obtain%
\begin{eqnarray*}
\left\langle (y,y^{\ast }),({x^{\ast }},x)\right\rangle -h^{\prime
}(y,y^{\ast }) &=&\left\langle y,x^{\ast }\right\rangle +\left\langle
x,y^{\ast }\right\rangle -\left\langle y,y^{\ast }\right\rangle -\eta \\
&=&\left\langle x,x^{\ast }\right\rangle -\left\langle x-y,x^{\ast }-y^{\ast
}\right\rangle -\eta \leq \left\langle x,x^{\ast }\right\rangle +\epsilon \\
&=&h(x,{x^{\ast }}).
\end{eqnarray*}%
Since $(y,y^{\ast })\in \mathrm{dom}\left( h^{\prime }\right) $ is
arbitrary, we can take supremum in the left hand side to obtain $h^{\prime
\ast }({x^{\ast }},x)\leq h(x,{x^{\ast }}),$ which, taking conjugates,
yields $h^{\ast }\circ i\leq h^{\prime }$. Conversely, assume that $h^{\ast
}\circ i\leq h^{\prime }$. Take ${x^{\ast }}\in E(\epsilon
,x)=L^{h}(\epsilon ,x)$ and ${y^{\ast }}\in E^{\prime }(\eta
,y)=L^{h^{\prime }}(\eta ,y)$. Using the assumption, together with these
inclusions and Fenchel-Young inequality, we get%
\begin{eqnarray*}
&&\left\langle (x,x^{\ast }),(y^{\ast },y)\right\rangle \\
&\leq &h(x,x^{\ast })+\left( h^{\ast }\circ i\right) (y,y^{\ast }) \\
&\leq &h(x,x^{\ast })+h^{\prime \ast }(y,y^{\ast })\leq \left\langle
x,x^{\ast }\right\rangle +\left\langle y,y^{\ast }\right\rangle +\epsilon
+\eta .
\end{eqnarray*}%
Re-arranging the left-most and right-most expressions we obtain \eqref{eq:4}%
. The last statement follows by taking $E^{\prime }=E$ in (i).\newline
(ii) The proof is based on Remark \ref{rem:add2}. Indeed, consider the set $%
S(h)$ given in Definition \ref{def2:additivity}. We claim that $L^{h}$ is
max-add if and only if $h\in S(h)$ and $h$ is minimal in $S(h)$. If the
claim is true, then Remark \ref{rem:add2} readily gives $h^{\ast }\circ i=h$%
. Let us proceed to prove the claim. Indeed, assume first that $L^{h}$ is
max-add. By Remark \ref{rem2:additivity}, we have $h\in S(h)$. It remains to
show that $h$ is minimal in $S(h)$. Let $h^{\prime }\in S(h)$ be such that $%
h^{\prime }\leq h$. We must show that $h^{\prime }=h$. Since $h^{\prime }\in
S(h),$ we have $h^{\prime \ast }\circ i\leq h^{\prime },$ and hence $%
L^{h^{\prime }}$ is additive by Remark \ref{rem2:additivity}. Since $%
h^{\prime }\leq h,$ we have that $L^{h}(\epsilon ,x)\subset L^{h^{\prime
}}(\epsilon ,x)$ for all $\epsilon \geq 0$ and all $x\in X$. Using now the
fact that $L^{h}$ is max-add and $L^{h^{\prime }}$ is additive, we conclude
that $L^{h}=L^{h^{\prime }}$. Given any enlargement $E\in \mathbb{E}(T)$,
the map from $E$ to $h_{E}$ is a bijection. This fact, together with the
equality $L^{h}=L^{h^{\prime }}$, allows us to conclude that $h=h^{\prime }$%
. Hence $h$ is minimal in $S(h)$. Conversely, assume that $h\in S(h)$ and
that $h$ is a minimal element of $S(h)$. Let $h^{\prime }\in \mathcal{H}%
_{a}(T)$ be such that $L^{h}(\epsilon ,x)\subset L^{h^{\prime }}(\epsilon
,x) $ for all $\epsilon \geq 0$ and all $x\in X$. This implies that $%
h^{\prime }\leq h$;\ indeed, if $(x,x^{\ast })\in \mathrm{dom}\left(
h\right) $ then, setting $\epsilon :=h(x,{x^{\ast }})-\left\langle x,x^{\ast
}\right\rangle \geq 0,$ we have $(x,x^{\ast })\in L^{h}(\epsilon ,x)$ and
hence $(x,x^{\ast })\in L^{h^{\prime }}(\epsilon ,x),$ that is,%
\begin{equation*}
h^{\prime }\left( x,x^{\ast }\right) \leq \left\langle x,x^{\ast
}\right\rangle +\epsilon =h(x,{x^{\ast }}).
\end{equation*}%
Moreover, since $h^{\prime }\in \mathcal{H}_{a}(T),$ by Remark \ref%
{rem1:additivity} we have that $h^{\prime \ast }\circ i\leq h^{\prime },$
and hence $h^{\prime }\in S(h)$. The minimality of $h$ now implies that $%
h=h^{\prime }$. In other words, $L^{h}=L^{h^{\prime }}$ and therefore $L^{h}$
is max-add. This completes the proof of the claim. As mentioned above, now
(ii) follows directly from the claim and Remark \ref{rem:add2}. The fact
that $E$ is mutually additive with $L^{h^{\ast }\circ i}$ follows by taking $%
h^{\prime }:=h^{\ast }\circ i$ in part (i). By (i), $L^{h^{\ast }\circ i}$
is the largest of all enlargements mutually additive with $E$. By (ii), $E$
is max add if and only if $h=h^{\ast }\circ i$. Equivalently, $%
L^{h}=L^{h^{\ast }\circ i}$ and hence $E$ coincides with its additive
complement. This completes the proof. 
\end{proof}

\begin{Rem}
Using Zorn's lemma, it was proved in \cite{S03} that there exists $h\in 
\mathcal{H}(T)$ such that $h^{\ast }\circ i=h$, and hence there are max-add
elements in the family $\mathbb{E}(T)$. Other non-constructive examples of
autoconjugate convex representations of $T$ can be found in \cite%
{penot03,penot}. It is then natural to ask for a constructive example.
Indeed, in the case when we are provided with a convex representation $h$ of 
$T$ whose domain coincides with the domain of $h^{\ast }\circ i,$ we can
constructively obtain both a max-add enlargement by means of Corollary \ref%
{cor:Th} below and an autoconjugate convex representation of $T$ (however
the coincidence of those domains is an essential condition for having such a
possibility, as Example \ref{not auto} shows). Such convex representations
can be found in \cite{PZ} and \cite{BWY10}. The one found in \cite{PZ}
requires a mild constraint qualification, namely, that the affine hull of
the domain of $T$ is closed. The other ones do not require any constraint
qualification. Corollary \ref{cor:JE} gives an alternative non-constructive
proof of the existence of autoconjugate convex representations in the case
of operators for which a suitable convex representation is available.
\end{Rem}

It was shown in \cite{BS02} that there is a largest element in the family $%
\mathcal{H}(T)$, which we denote here by $\sigma _{T}$. It is shown in \cite%
{BS02} that $\sigma _{T}=\mathrm{cl}\,\mathrm{conv}(\pi +\delta _{G(T)})$,
where the notation $\delta _{G(T)}$ is used for the indicator function of
the graph of $T$. Moreover, according to \cite[eq. $(9)$]{BS02} we have $%
\sigma _{T}=\left( \mathcal{F}_{T}\right) ^{\ast }\circ i$, and this
function characterizes the smallest enlargement, i.e., $T^{SE}=L^{\sigma
_{T}}$.

We recover a result from \cite{SV} as a corollary of Proposition \ref{pro:MA}%
.

\begin{Cor}
\label{mut add}The biggest enlargement $T^{\mathrm{BE}}$ and the smallest
enlargement $T^{\mathrm{SE}}$ are mutually additive.
\end{Cor}

\begin{proof}
This follows from the fact that $T^{\mathrm{BE}}=L^{\mathcal{F}_{T}}$ and $%
T^{\mathrm{SE}}=L^{\sigma _{T}},$ together with the equality $\sigma
_{T}=\left( \mathcal{F}_{T}\right) ^{\ast }\circ i$.
\end{proof}

\bigskip

We next show that all members of our family are additive.

\begin{Cor}
\label{cor:3.22} \label{cor:Th} For every $h\in \mathcal{H}(T)$ we have that 
$\breve{T}_{h}$ is additive. The enlargement $\breve{T}_{h}$ is max-add if
and only if $\left( \mathcal{A}h\right) ^{\ast }\circ i=\mathcal{A}{h}$.
Consequently, if $h^{\ast }\circ i=h$ then $\breve{T}_{h}$ is max-add.
\end{Cor}

\begin{proof}
Recall (Proposition \ref{pro:Lh}) that 
\begin{equation}
\breve{T}_{h}=L^{\mathcal{A}h}.  \label{eq:11}
\end{equation}%
By Proposition \ref{pro:MA}(i), it is enough to prove that $\left( {\mathcal{%
A}{h}}\right) ^{\ast }\circ i\leq \mathcal{A}h$, which is precisely the
conclusion of the first assertion in Theorem \ref{theo:JE}(i). This proves
the first statement. The second statement follows from (\ref{eq:11}) and
Proposition \ref{pro:MA}(ii). 
If $h={h}^{\ast }\circ i$ then we have $\mathcal{A}h=\left( {\mathcal{A}{h}}%
\right) ^{\ast }\circ i$. Indeed, if $h={h}^{\ast }\circ i$ it is direct to
check that $\mathcal{A}h=h$. So $\left( {\mathcal{A}{h}}\right) ^{\ast
}\circ i={h}^{\ast }\circ i=h=\mathcal{A}h$. By (\ref{eq:11}) and
Proposition \ref{pro:MA}(ii), we conclude that $\breve{T}_{h}$ is max-add.
\end{proof}

\bigskip

A consequence of the above results is that, besides the smallest
enlargement, a whole subfamily of enlargements happens to be additive. If
they derive from an autoconjugate $h$, then they are max-add and hence they
can be regarded as \textquotedblleft structurally closer\textquotedblright\
to the epsilon subdifferential. We will see in the next section that a
particular member of this subfamily is precisely the $\epsilon $%
-subdifferential when $T=\partial f$.


\section{The case $T:=\partial f$}

We want now to establish the relation between our new enlargement and the $%
\epsilon $-subdifferential in the case $T:=\partial f.$

\begin{Lem}
\label{h leq FY}Let $f:X\rightarrow \mathbb{R}\cup \{+\infty \}$ be a
lower semicontinuous proper convex function, and let $h\in \mathcal{H}(\partial f)$ be
such that $h\leq f^{FY}.$ Then, for every $(x,x^{\ast })\in X\times X^{\ast
},$ one has 
\begin{equation*}
h^{FY}\left( (x,x^{\ast }),(x^{\ast },x)\right) \geq \left\langle x,x^{\ast
}\right\rangle +f^{FY}(x,x^{\ast }).
\end{equation*}
\end{Lem}

\begin{proof}
From the inequality $h\leq f^{FY}$ it follows that $h^{\ast }\geq \left(
f^{FY}\right) ^{\ast }=f^{FY}\circ i;$\ hence, using that $h(x,x^{\ast
})\geq \left\langle x,x^{\ast }\right\rangle ,$ we obtain%
\begin{equation*}
h^{FY}\left( (x,x^{\ast }),(x^{\ast },x)\right) =h(x,x^{\ast })+h^{\ast
}(x^{\ast },x)\geq \left\langle x,x^{\ast }\right\rangle +f^{FY}(x,x^{\ast
}).
\end{equation*}
\end{proof}

\begin{Thm}
\label{Theo1} Let $f$ and $h$ be as in Lemma \ref{h leq FY}$.$ Then, for
every $\epsilon >0$ and $x\in X,$ one has%
\begin{equation*}
\breve{T}_{h}(\frac{\epsilon }{2},x)\subset \breve{{\partial }}f(\epsilon ,x)
\end{equation*}
\end{Thm}

\begin{proof}
Let $x^{\ast }\in \breve{T}_{h}(\frac{\epsilon }{2},x)=L^{\mathcal{A}h}(\frac{\epsilon }{2},x).$ Then, by Lemma \ref%
{h leq FY}, we have%
\begin{equation*}
\begin{array}{rcl}
f^{FY}(x,x^{\ast })&\leq& h^{FY}\left( (x,x^{\ast }),(x^{\ast },x)\right)
-\left\langle x,x^{\ast }\right\rangle \\
&&\\
&=& 2 \mathcal{A}h(x,x^{\ast}) -\left\langle x,x^{\ast }\right\rangle
\\
&&\\
&\leq& 2\left( \left\langle x,x^{\ast
}\right\rangle +\frac{\epsilon }{2}\right) -\left\langle x,x^{\ast
}\right\rangle =\left\langle x,x^{\ast }\right\rangle +\epsilon,
\end{array}
\end{equation*}%
where we used the definition of $\mathcal{A}$ in the first equality, and the
assumption on $x^{\ast}$ in the second inequality. This shows that $x^{\ast }\in \breve{{\partial }}f(\epsilon ,x).$ 
\end{proof}

\begin{Rem}
\label{rem2}As observed in Remark \ref{R0}, when $\epsilon =0$ in Theorem %
\ref{Theo1}(ii) we recover the equality $T_{h}=\partial f,$ proved in \cite[%
Example 2.3]{F}.
\end{Rem}

When $h:={{\mathcal{F}}}_{\partial f}$, we can strengthen the inclusion in
Theorem \ref{Theo1}:

\begin{Pro}
\label{pro4.3} Let $f:X\rightarrow \mathbb{R}\cup \{+\infty \}$ be a  lower semicontinuous 
proper convex function. Then, for every $\epsilon \geq 0$ and $x\in X,$ we
have 
\begin{equation*}
\breve{T}_{{\mathcal{F}}_{\partial f}}(\frac{\epsilon }{2},x)\subset T^{%
\mathrm{SE}}(\epsilon ,x).
\end{equation*}
\end{Pro}

\begin{proof}
Suppose that $x^{\ast }\in \breve{T}_{{\mathcal{F}}_{\partial f}}(\frac{%
\epsilon }{2},x)$. Then we can write 
\begin{equation*}
\frac{1}{2}({\mathcal{F}}_{\partial f}+{\mathcal{F}}_{\partial f}^{\ast
}\circ i)(x,x^{\ast })\leq \langle x,x^{\ast }\rangle +\frac{\epsilon }{2}.
\end{equation*}%
Using the fact that ${\mathcal{F}}_{\partial f}\geq \langle \cdot ,\cdot
\rangle $, the last inequality yields 
\begin{equation}
\left( {\mathcal{F}}_{\partial f}^{\ast }\circ i\right) (x,x^{\ast })\leq
\langle x,x^{\ast }\rangle +\epsilon .
\end{equation}%
Equivalently, $x^{\ast }\in L^{{\mathcal{F}}_{\partial f}^{\ast }\circ
i}(\epsilon ,x)=T^{\mathrm{SE}}(\epsilon ,x)$ (see the proof of Corollary %
\ref{mut add}).
\end{proof}

\begin{Rem}
Since $h^{FY}$ is autoconjugate , we see that $\breve{T}_{h^{FY}}$ is
max-add. Indeed, this fact follows from Theorem \ref{Theo1} and \cite[%
Theorem 6.4]{SV} (see also Corollary \ref{cor:3.22}). Is this the only
max-add enlargement of $T:=\partial f$? The answer is no, since an example
in \cite{BWY10} shows three different autoconjugate convex representations
of a subdifferential operator, which result in three different max-add
enlargements.
\end{Rem}
\begin{acknowledgement}
  The authors would like to thank the two anonymous referees for their
  valuable comments and suggestions, which have led to an improved
  paper. The authors are very grateful to Heinz Bauschke and Benar Fux
  Svaiter for their comments and corrections on an earlier version of
  this manuscript. Bauschke kindly indicated to us an additional
  reference for Remark 2.5, while Svaiter kindly indicated that the
  results in \cite{S03} constitute the earliest proof of the "only if" part
  of Theorem 2.2.

\end{acknowledgement}
\bibliographystyle{plain}
\bibliography{myrefsBMLRT-3.bib}

\end{document}